\documentclass[english,11pt]{article}

\usepackage[english]{babel}
\usepackage[cp850]{inputenc}
\usepackage{latexsym,graphicx, fancybox}
\usepackage{amssymb, amsmath, amsfonts}
\usepackage{color}
\usepackage{latexsym}

\font\tenmath=msbm10 scaled 1200
\font\sevenmath=msbm7 scaled 1200
\font\fivemath=msbm5 scaled 1200 

\newfam\mathfam \textfont\mathfam=\tenmath
\scriptfont\mathfam=\sevenmath \scriptscriptfont\mathfam=\fivemath

\def\R{{\mathbb R}}
\def\N{{\mathbb N}}
\def\E{{\mathbb E}}

\def\P{{\mathbb P}}

\newtheorem{Theorem}{Theorem}[section]
\newtheorem{Definition}[Theorem]{Definition}
\newtheorem{Proposition}[Theorem]{Proposition}

\newtheorem{Lemma}[Theorem]{Lemma}
\newtheorem{Corollary}[Theorem]{Corollary}
\newtheorem{Remark}[Theorem]{Remark}

\numberwithin{equation}{section}

\def \^#1{\if#1i{\accent"5E\i}\else{\accent"5E#1}\fi}

\def \ind {1 \mkern -5mu \hbox{I}}

\begin{document}
\selectlanguage{english}
\title{\bf Coupling-convexity of a diffusion with respect to its starting position
}
 
\author{ 
 {\sc Benjamin Jourdain}\thanks{CERMICS, ENPC, CNRS, Institut Polytechnique de Paris, Inria,  Marne-la-Vall\'ee, France. E-mail: {\tt   benjamin.jourdain@enpc.fr}}~\thanks{This research
benefited from the support of the ``Chaire Risques Financiers'', Fondation du Risque.\\Anthropic's Claude Opus 5.5 has been used for proof checking and language improvement at a final stage of the preparation of this manuscript.}}

\maketitle
\begin{abstract}
   In the present paper, we consider a $d$-dimensional time-homogeneous stochastic differential equation with diffusion matrix $a$ and no drift coefficient. We check that the convexity in the starting position $x$ of the distribution of its solution on the path space when the codomain is endowed with the convex order is equivalent to the convexity of $\R^d\ni x\mapsto \sqrt{\zeta^*a(x)\zeta}$ for each $\zeta\in\R^d$. To do so, we introduce the seemingly stronger notion of coupling-convexity in the starting position and check that it is also equivalent to the convexity of the above square-rooted functions.
 \end{abstract}
 Studies on the propagation of convexity by stochastic dynamics and convex ordering for stochastic processes go back more than fourty years to a seminal work by Hajek~\cite{Hajek} (for increasing convex ordering). It was followed in the late 1990s by El Karoui, Jeanblanc and Shreve~\cite{EKJS},  Martini~\cite{Martini} and  then by various authors, among them Lions and Musiela~\cite{LM06}, Schied and Stadje~\cite{SS}, Bergenthum and R\"uschendorf in a series of papers~\cite{Ruschendorf1, Ruschendorf2,Ruschendorf3}  or, more recently, Pag\`es~\cite{Pag2016} (for various stochastic dynamics), Liu and Pag\`es~\cite{LiuPagAAP,LiuPag} (for McKean--Vlasov processes), Jourdain and Pag\`es~\cite{JourPagVolt22,JPMOR} (respectively for Volterra processes and one-dimensional stochastic differential equations). Except for Hajek~\cite{Hajek}, most of these papers were   motivated  by   financial applications to the pricing and hedging of derivative products when the underlying asset price evolves according to the dynamics under consideration.
 
The present paper is dedicated to rather simple dynamics: we consider a
$d$-dimensional time-homogeneous stochastic differential equation with no drift coefficient (the extension to a time-homogeneous affine drift coefficient is discussed in Remark \ref{remmt}). The diffusion matrix $a$ is a mapping from $\R^d$ to the set of symmetric positive semi-definite $d\times d$-matrices, which we suppose to be continuous with at most quadratic growth. We work under assumptions ensuring that the associated martingale problem is well-posed and, for $(t,x)\in\R_+\times\R^d$, denote respectively by $P^x_{\le t}$, $P^x_t$ and $P^x_{0,t}$  its solution on the time interval $[0,t]$ started from $x$ and the marginals of this solution at time $t$ and times $0,t$. The propagation of convexity means that for each convex function $\varphi:\R^d\to\R$, $\R^d\ni x\mapsto\int_{\R^d}\varphi(z)P^x_t(dz)$ is convex for each $t\ge 0$. It amounts to the convexity of $\R^d\ni x\mapsto P^x_t$, where the codomain is endowed with the convex order between probability measures. According to an argument based on the maximum principle for the associated Feynman--Kac partial differential equation \cite{LM06}, when there exists a Lipschitz function $\sigma:\R^d\to\R^{d\times q}$ with $q\in\N^*$ such that $a(x)=\sigma\sigma^*(x)$ for each $x\in\R^d$, the propagation of convexity is equivalent to the convexity of $\R^d\ni x\mapsto \sqrt{\zeta^*a\left(x+(\lambda-\zeta^*x)\zeta\right)\zeta}$ for each $(\zeta,\lambda)\in\R^d\times \R$ such that $|\zeta|=1$ (see Remark \ref{remcnlm} below for a simple proof of the necessary condition). When $d=1$, the latter condition and therefore the propagation of convexity always holds, which can be proved by a very nice probabilistic argument \cite{EKJS}. Still when $d=1$, according to \cite[Theorem 1 and Proposition 2]{JPMOR}, the convexity of $\R\ni x\mapsto P^x_{\le t}$ for each $t\ge 0$ is equivalent to the convexity of $\R\ni x\mapsto P^x_{0,t}$ for each $t\ge 0$ and to the convexity of $\R\ni x\mapsto \sqrt{a(x)}$. The main result in the present paper (see Theorem \ref{mainthm} below) is the following extension to any dimension $d$: the convexity of $\R^d\ni x\mapsto P^x_{\le t}$ for each $t\ge 0$ is equivalent to the convexity of $\R^d\ni x\mapsto P^x_{0,t}$ for each $t\ge 0$ and to the convexity of $\R^d\ni x\mapsto \sqrt{\zeta^*a(x)\zeta}$ for each $\zeta\in\R^d$. A key ingredient in the proof is the very nice recent result by Gu\'eant \cite{gueant} showing that the sufficient condition and the necessary condition given in \cite{JPJMVA} for the convex domination of a Gaussian distribution by a convex combination of two such distributions are equivalent. The reason why we rely on convex comparison of Gaussian mixtures is that we first deal with Euler discretizations of the stochastic differential equation before taking the limit as the time-step goes to $0$, a strategy that has proved successful in establishing propagation of convexity in \cite{BRFS,BRAP,PSP,LiuPag,LiuPagAAP,JourPagVolt22,JPMOR}. The result by Gu\'eant \cite{gueant} ensures that the convexity of $\R^d\ni x\mapsto \sqrt{\zeta^*a(x)\zeta}$ for each $\zeta\in\R^d$ implies that for each $(x,y,\alpha)\in\R^d\times\R^d\times[0,1]$, there exists 
$\Theta\in\R^{d\times d}$ such that
   \begin{equation*}
      \left(\begin{array}{cc} a(x) & \Theta \\ \Theta^* & a(y)\end{array}\right)\in{\cal S}_+(2d)\mbox{ and }a(\alpha x+(1-\alpha)y)\le \alpha^2 a(x)+\alpha(1-\alpha)(\Theta+\Theta^*)+(1-\alpha)^2a(y).\label{theta}
    \end{equation*}
This condition actually yields that ${\cal N}_{d}\left(\alpha x+(1-\alpha)y
  ,a(\alpha x+(1-\alpha)y)t\right)$ is smaller in the convex order than
the image of ${\cal N}_{2d}\left(\left(\begin{array}{c}x
   \\y
\end{array}\right),\left(\begin{array}{cc}a(x) & \Theta
   \\ \Theta^*& a(y)
                         \end{array}\right)t\right)$ by $\R^d\times\R^d\ni(\tilde x,\tilde y)\mapsto \alpha \tilde x+(1-\alpha)\tilde y\in\R^d$. Note that the $2d$-dimensional Gaussian distribution is a coupling between ${\cal N}_{d}\left(x
  ,a(x)t\right)$ and ${\cal N}_{d}\left(y
  ,a(y)t\right)$. This leads us to similarly define the existence of a convex order coupling between a probability measure and a convex combination of two such measures. This notion naturally extends into coupling-convexity of $\R^d\ni x\mapsto P^x_t$: for each $(x,y,\alpha)\in\R^d\times\R^d\times [0,1]$, there exists a coupling between $P^x_t$ and $P^y_t$, the image of which by $\R^d\times\R^d\ni (\tilde x,\tilde y)\mapsto \alpha\tilde x+(1-\alpha)\tilde y$ dominates $P^{\alpha x+(1-\alpha) y}_t$ in the convex order. The coupling-convexity of $\R^d\ni x\mapsto P^x_{\le t}$ is defined in a similar way. It turns out that the coupling-convexity of $\R^d\ni x\mapsto P^x_t$ is equivalent to the convexity of $\R^d\ni x\mapsto P^x_{0,t}$ (Proposition \ref{propeqcvccv}) and the coupling-convexity and the convexity of $\R^d\ni x\mapsto P^x_{\le t}$ are equivalent (Remark \ref{remeqcvccv}). To prove our main result, we check that the convexity of $\R^d\ni x\mapsto \sqrt{\zeta^*a(x)\zeta}$ for each $\zeta\in\R^d$ is equivalent to the coupling-convexity of $\R^d\ni x\mapsto P^x_{\le t}$.

Section \ref{secprm} is devoted to preliminaries: we recall the definition of the Wasserstein spaces of probability measures and the convex order between a probability measure and a convex combination of two such measures before introducing the stronger notion of existence of a convex order coupling and remarking that the two notions coincide for Gaussian distributions according to \cite{gueant}. In Section \ref{secmr}, we state our main result, namely the equivalence of the coupling-convexity of $x\mapsto P^x_{\le t}$ for each $t\ge 0$ with the convexity of $\R^d\ni x\mapsto \sqrt{\zeta^*a(x)\zeta}$ for each $\zeta\in\R^d$. We also state equivalence with other assertions which turns out to be useful in the proof. We deduce a sufficient condition for the convex comparison of the distributions of the solutions to two stochastic differential equations. The proofs are given in Section \ref{secproofs}.\par

\noindent {\sc Notation}
\begin{itemize}
\item Let ${\cal S}_+(d)$ denote the set of symmetric positive semi-definite $d\times d$-matrices. We denote by $M^{1/2}\in{\cal S}_+(d)$ the symmetric square root of $M\in{\cal S}_+(d)$ and, when $M$ is non-singular, by $M^{-1/2}$ the symmetric square root of its inverse $M^{-1}$.\item For $z\in\R^d$ and $M\in{\cal S}_+(d)$, we denote by ${\cal N}_d(z,M)$ the $d$-dimensional Gaussian distribution with expectation $z$ and covariance matrix $M$.\item For $M,N\in{\cal S}_+(d)$, we write $M\le N$ when $N-M\in{\cal S}_+(d)$.
\item Let ${\mathcal O}(d)$ denote the set of $d\times d$ orthogonal matrices.

\item We denote by $M^*$ the transpose of a matrix $M$.
\item For two probability measures $\mu$ and $\nu$ on Polish spaces, we denote by $\Pi(\mu,\nu)$ the set of couplings between $\mu$ and $\nu$ i.e. of probability measures on the product space with respective marginals $\mu$ and $\nu$. 
  \item We denote by ${\mathbb C}(Z)\in{\cal S}_+(d)$ the covariance matrix of a square-integrable $d$-dimensional random vector $Z$.
\end{itemize}

\section{Preliminaries on Wasserstein spaces, convex order and convex order couplings}\label{secprm}
\subsection{Wasserstein spaces}Let us recall the Wasserstein space  with index $\rho\in[1,+\infty)$ on a Polish space ${\cal Z}$ for the metric $d_{\cal Z}$:
$${\cal P}_\rho({\cal Z})=\left\{\mu\mbox{ probability measure on ${\cal Z}$ such that }\forall z_0\in{\cal Z},\;\int_{\cal Z}d_{\cal Z}(z_0,z)^\rho \mu(dz)<\infty\right\}.$$
Note that, in view of the triangle inequality, we could replace $\forall z_0\in{\cal Z}$ by $\exists z_0\in{\cal Z}$ in this definition and we now fix some $z_0\in{\cal Z}$. We endow ${\cal P}_\rho({\cal Z})$ with the Wasserstein metric
$${\cal W}_\rho(\mu,\nu)=\left(\inf_{\pi\in\Pi(\mu,\nu)}\int_{{\cal Z}\times{\cal Z}}d_{\cal Z}(z,\tilde z)^\rho\pi(dz,d\tilde z)\right)^{1/\rho}.$$
Of course, $\rho\mapsto {\cal W}_\rho$ is non-decreasing.
According to \cite[Theorem 6.9]{villani}, for a sequence $(\mu_n)_{n\in\N}$ of elements of ${\cal P}_\rho({\cal Z})$ and $\mu\in{\cal P}_\rho({\cal Z})$, we have, as $n\to\infty$,
\begin{align}
   {\cal W}_\rho(\mu_n,\mu)\rightarrow 0\Leftrightarrow \mu_n\rightarrow \mu \mbox{ weakly and }\int_{\cal Z}d_{\cal Z}(z_0,z)^\rho \mu_n(dz)\rightarrow \int_{\cal Z}d_{\cal Z}(z_0,z)^\rho \mu(dz).\label{cvwr}
\end{align}
The Wasserstein distance with index $1$ admits the following dual formulation:
\begin{equation}
   {\cal W}_1(\mu,\nu)=\sup_{\varphi:{\cal Z}\to\R\mbox{ $1$-Lipschitz}}\left\{\int_{\cal Z}\varphi(z)\mu(dz)-\int_{\cal Z}\varphi(z)\nu(dz)\right\}\label{duw1}.
 \end{equation}
\subsection{Convex order}
 \begin{Definition}

   \begin{itemize}
   \item For $\mu,\nu\in{\cal P}_1(\R^d)$, we say that $\mu$ is smaller than $\nu$ in the convex order and write $\mu\le_{\rm cx}\nu$ if
     $$\forall\varphi:\R^d\to\R\mbox{ convex},\;\int_{\R^d}\varphi(y)\mu(dy)\le\int_{\R^d}\varphi(y)\nu(dy).$$
   \item For $t\in(0,+\infty)$ and $P,Q\in{\cal P}_1(C([0,t],\R^d))$ where $C([0,t],\R^d)$ is endowed with the sup norm, we say that $P$ is smaller than $Q$ in the convex order and write $P\le_{\rm cx}Q$ if 
     $$\forall\phi:C([0,t],\R^d)\to\R\mbox{ l.s.c. convex},\;\int_{C([0,t],\R^d)}\phi(\xi)P(d\xi)\le\int_{C([0,t],\R^d)}\phi(\xi)Q(d\xi).$$
   \item When $X$ and $Y$ are integrable $\R^d$-valued random vectors or $C([0,t],\R^d)$-valued processes, we also denote $X\le_{\rm cx}Y$ when the distribution of $X$ is smaller than that of $Y$ in the convex order.
   \end{itemize}
 \end{Definition}
According to Strassen's theorem, for $\mu,\nu\in{\cal P}_1(\R^d)$, the convex domination $\mu\le_{\rm cx}\nu$ is equivalent to the existence of a martingale coupling $\pi\in\Pi(\mu,\nu)$ i.e. a coupling with disintegration $\pi(dx,dy)=\mu(dx)\pi_x(dy)$ such that $\mu(dx)$ a.e., $\int_{\R^d}y\pi_x(dy)=x$. The next lemma ensures that it is enough to restrict to Lipschitz convex test functions $\varphi$ or functionals $\phi$ to characterize the convex order, which can be proved by inf-convolution (see for instance \cite[Lemmas 2.9 and 2.11]{JPFS}). It also deals with consequences of the convex order between $\mu,\nu\in{\cal P}_2(\R^d)$ in terms of their covariance matrices.
 \begin{Lemma}\label{lemcvord}\begin{description}
    \item[(1)]For $\mu,\nu\in{\cal P}_1(\R^d)$, $$\mu\le_{\rm cx}\nu\Leftrightarrow \forall\varphi:\R^d\to\R\mbox{ Lipschitz convex},\;\int_{\R^d}\varphi(y)\mu(dy)\le\int_{\R^d}\varphi(y)\nu(dy).$$
    \item[(2)]When $\mu,\nu\in{\cal P}_2(\R^d)$ with respective covariance matrices ${\mathbb C}_\mu$ and ${\mathbb C}_\nu$ satisfy $\mu\le_{\rm cx}\nu$, then ${\mathbb C}_\mu\le {\mathbb C}_\nu$. Conversely, ${\mathbb C}_\mu\le {\mathbb C}_\nu\Rightarrow \forall y\in\R^d,\;{\cal N}_d(y,{\mathbb C}_\mu)\le_{\rm cx} {\cal N}_d(y,{\mathbb C}_\nu)$.
   \item[(3)]For $P,Q\in{\cal P}_1(C([0,t],\R^d))$, $$P\le_{\rm cx}Q\Leftrightarrow \forall\phi:C([0,t],\R^d)\to\R\mbox{ Lipschitz convex},\;\int_{C([0,t],\R^d)}\phi(\xi)P(d\xi)\le\int_{C([0,t],\R^d)}\phi(\xi)Q(d\xi).$$
  \end{description}
\end{Lemma}
To prove {\bf (2)}, we remark that for $\mu,\nu\in{\cal P}_1(\R^d)$, $\mu\le_{\rm cx}\nu\Rightarrow \int_{\R^d}y\mu(dy)=\int_{\R^d}y\nu(dy)$. When moreover $\mu,\nu\in{\cal P}_2(\R^d)$, then the choice $\varphi_\xi(x)=(\xi^*(x-\int_{\R^d}y\mu(dy)))^2$ where $\xi\in\R^d$ implies that $\xi^*{\mathbb C}_\mu\xi=\int_{\R^d}\varphi_\xi(x)\mu(dx)\le \int_{\R^d}\varphi_\xi(y)\nu(dy)=\xi^*{\mathbb C}_\nu\xi$ so that ${\mathbb C}_\mu\le {\mathbb C}_\nu$. When conversely ${\mathbb C}_\mu\le {\mathbb C}_\nu$, then for $X\sim{\cal N}_d(y,{\mathbb C}_\mu)$ and $Z\sim{\cal N}_d(0,{\mathbb C}_\nu-{\mathbb C}_\mu)$ independent, $Y=X+Z\sim{\cal N}_d(y,{\mathbb C}_\nu)$ and by Jensen's inequality and the tower property, for $\varphi:\R^d\to\R$ convex, $\E[\varphi(X)]=\E[\varphi(\E[Y|X])]\le \E[\E[\varphi(Y)|X]]=\E[\varphi(Y)]$.
\subsection{Existence of convex order couplings}
\begin{Definition}
   For $\eta,\mu,\nu\in{\cal P}_1(\R^d)$ (resp. ${\cal P}_1(C([0,t],\R^d))$) and $\alpha\in[0,1]$, we say that ${\rm COC}(\eta,\mu,\nu,\alpha)$ holds if there exists $\pi\in{\Pi}(\mu,\nu)$ with image $\tilde\pi^\alpha$ by (resp. $C([0,t],\R^d)\times C([0,t],\R^d)$) $\R^d\times \R^d\ni(x,y)\mapsto \alpha x+(1-\alpha)y$ such that $\eta\le_{\rm cx}\tilde\pi^\alpha$.
\end{Definition}
Since for $\varphi$ convex, $\varphi(\alpha x+(1-\alpha)y)\le \alpha\varphi(x)+(1-\alpha)\varphi(y)$, we have \begin{equation}
   {\rm COC}(\eta,\mu,\nu,\alpha)\Rightarrow \eta\le_{\rm cx}\alpha\mu+(1-\alpha)\nu.\label{COCimpC}
 \end{equation}When $\alpha\in\{0,1\}$, $\eta\le_{\rm cx}\alpha\mu+(1-\alpha)\nu$ implies $\eta\le_{\rm cx}\tilde\pi^\alpha$ for each $\pi\in\Pi(\mu,\nu)$, so that the two sides of the implication \eqref{COCimpC} are equivalent. 
 Before investigating the gap between them in the case $\alpha\in(0,1)$ and  $\eta,\mu,\nu\in{\cal P}_1(\R^d)$, let us check that the existence of a convex order coupling is equivalent to a stronger property than mere convex ordering.
 \begin{Lemma}\label{lemcoc2marg}
   For $\eta,\mu,\nu\in{\cal P}_1(\R^d)$ and
   $\alpha\in[0,1]$, we have 
 \begin{align*}
   {\rm COC}(\eta,\mu,\nu,\alpha)\Leftrightarrow &\forall q\in\N^*,\;\forall x,y\in\R^q,\;\delta_{\alpha x+(1-\alpha)y}\otimes \eta\le_{\rm cx}\alpha\delta_{x}\otimes \mu+(1-\alpha)\delta_{y}\otimes\nu\\
   \Leftrightarrow&\exists q\in\N^*,\;\exists x,y\in\R^q\mbox{ distinct},\;\delta_{\alpha x+(1-\alpha)y}\otimes \eta\le_{\rm cx}\alpha\delta_{x}\otimes \mu+(1-\alpha)\delta_{y}\otimes\nu.
 \end{align*}
 \end{Lemma}
 \noindent {\bf Proof of Lemma \ref{lemcoc2marg}.}
We are going to prove that the first assertion implies the second and that the third, which is an obvious consequence of the second, implies the first. Let us assume  ${\rm COC}(\eta,\mu,\nu,\alpha)$ and let $\pi\in\Pi(\mu,\nu)$ be such that $\eta\le_{\rm cx}\tilde\pi^\alpha$, $q\in\N^*$, $x,y\in\R^q$ and $\varphi:\R^q\times\R^d\to\R$ be convex. We have\begin{align*}
   \int_{\R^q\times\R^d}\varphi(z,w)\delta_{\alpha x+(1-\alpha)y}&\otimes \eta(dz,dw)=\int_{\R^d}\varphi(\alpha x+(1-\alpha)y,w)\eta(dw)\\&\le 
 \int_{\R^d\times\R^d}\varphi(\alpha x+(1-\alpha)y,\alpha \tilde x+(1-\alpha)\tilde y)
   \pi(d\tilde x,d\tilde y) \\&\le\int_{\R^d\times\R^d}\left(\alpha\varphi( x,\tilde x)+(1-\alpha)\varphi(y,\tilde y)\right)\pi(d\tilde x,d\tilde y)\\&=\alpha\int_{\R^d}\varphi( x,\tilde x)\mu(d\tilde x)+(1-\alpha)\int_{\R^d}\varphi( y,\tilde y)\nu(d\tilde y)\\&=\int_{\R^q\times\R^d}\varphi(z,w)\left(\alpha\delta_{x}\otimes\mu+(1-\alpha)\delta_{y}\otimes\nu\right)(dz,dw).
                                                                                                                                                                                                  \end{align*}
                                                                                                                                                                                                  Hence $\delta_{\alpha x+(1-\alpha)y}\otimes \eta\le_{\rm cx}\alpha\delta_{x}\otimes \mu+(1-\alpha)\delta_{y}\otimes\nu$.\par
                                                                                                                                                                                                  Let us now assume that this inequality holds with $x\neq y$ and $\alpha\in (0,1)$. Then, according to Strassen's theorem, there exists $\gamma\in\Pi(\delta_{\alpha x+(1-\alpha)y}\otimes \eta,\alpha\delta_{x}\otimes \mu+(1-\alpha)\delta_{y}\otimes\nu)$ which is a martingale coupling. Writing the disintegration $\gamma(dz,dw,d\tilde z,d\tilde w)=\delta_{\alpha x+(1-\alpha)y}(dz)\eta(dw)\gamma_w(d\tilde z,d\tilde w)$, the martingale property amounts to $\eta(dw)$ a.e., $\int_{\R^q\times\R^d} (\tilde z,\tilde w) \gamma_w(d\tilde z,d\tilde w)=(\alpha x+(1-\alpha)y,w)$ and implies that \begin{equation}
                                                                                                                                                                                                    \eta(dw)\mbox{ a.e.},\;\gamma_w(\{x\}\times\R^d)=\alpha=1-\gamma_w(\{y\}\times\R^d)\mbox{ and }\int_{\R^d}\tilde w\gamma_w(\R^q,d\tilde w)=w.\label{carcmart}\end{equation} Moreover, the fact that $\gamma\in\Pi(\delta_{\alpha x+(1-\alpha)y}\otimes \eta,\alpha\delta_{x}\otimes \mu+(1-\alpha)\delta_{y}\otimes\nu)$ with $x\ne y$ implies that $\int_{w\in\R^d}\gamma_w(\{x\},d\tilde x)\eta(dw)=\alpha\mu(d\tilde x)$ and $\int_{w\in\R^d}\gamma_w(\{y\},d\tilde x)\eta(dw)=(1-\alpha)\nu(d\tilde y)$. Also using the first two equalities in \eqref{carcmart}, we deduce that $\pi(d\tilde x,d\tilde y)=\int_{w\in\R^d}\frac{1}{\alpha(1-\alpha)}\gamma_w(\{x\},d\tilde x)\gamma_w(\{y\},d\tilde y)\eta(dw)$ belongs to $\Pi(\mu,\nu)$. For $\varphi:\R^d\to\R$ convex, Jensen's inequality implies that\begin{align*}\int_{\R^d\times\R^d}\varphi(\alpha \tilde x+(1-\alpha)\tilde y)\pi(d\tilde x,d\tilde y)&\ge \int_{\R^d}\varphi\left(\int_{\R^d\times\R^d}\Big(\frac{\tilde x}{1-\alpha}+\frac{\tilde y}{\alpha}\Big)\gamma_w(\{x\},d\tilde x)\gamma_w(\{y\},d\tilde y)\right)\eta(dw)\end{align*}
                                                                                                                                                                                                  Using the first two equalities in \eqref{carcmart} then the last equality, we get that $\eta(dw)$ a.e., 
\begin{align*}\int_{\R^d\times\R^d}\Big(\frac{\tilde x}{1-\alpha}+\frac{\tilde y}{\alpha}\Big)\gamma_w(\{x\},d\tilde x)\gamma_w(\{y\},d\tilde y)&=\int_{\R^d}\tilde x
   \gamma_w(\{x\},d\tilde x)+\int_{\R^d}\tilde y
   \gamma_w(\{y\},d\tilde y)\\&=\int_{\R^d}\tilde w\gamma_w(\R^q,d\tilde w)=w.
\end{align*}
Plugging this equality in the last inequality, we conclude that $\eta\le_{\rm cx}\tilde\pi^\alpha$ so that ${\rm COC}(\eta,\mu,\nu,\alpha)$ holds.                             \hfill$\Box$

                                                                                                                                                The next lemma also illustrates that the existence of a convex order coupling is stronger than mere convex ordering.\begin{Lemma}
   Let $\mu,\nu\in{\cal P}_1(\R^d)$ and $\alpha\in (0,1)$. Then ${\rm COC}(\alpha\mu+(1-\alpha)\nu,\mu,\nu,\alpha)\Rightarrow \mu=\nu.$
 \end{Lemma}
 \noindent {\bf Proof.} Let us assume ${\rm COC}(\alpha\mu+(1-\alpha)\nu,\mu,\nu,\alpha)$ with $\alpha\in(0,1)$ and let $\pi\in\Pi(\mu,\nu)$ be such that $\alpha\mu+(1-\alpha)\nu\le_{\rm cx}\tilde\pi^\alpha$. For $\varphi:\R^d\to \R$ convex with affine growth, we have
\begin{align*}
   0&\le \int_{\R^d}\varphi(z)\tilde\pi^\alpha(dz)-\int_{\R^d}\varphi(z)(\alpha\mu+(1-\alpha)\nu)(dz)\\&=\int_{\R^d\times\R^d}\left(\varphi(\alpha x+(1-\alpha)y)-(\alpha \varphi(x)+(1-\alpha)\varphi(y))\right)\pi(dx,dy).
\end{align*}
The integrand in the last integral is non-positive and even negative when $\varphi$ is strictly convex and $x\ne y$. Therefore, choosing the strictly convex function $\varphi(z)=\sqrt{1+|z|^2}$, we conclude that $\pi(dx,dy)$ a.e., $x=y$ so that $\mu=\nu$.
 \hfill$\Box$
                                                                                                                                                
Let us investigate the difference between the two sides of \eqref{COCimpC} in the case $\eta,\mu,\nu\in{\cal P}_1(\R)$ where the convex order can be characterized in terms of antiderivatives of quantile functions. For $\gamma\in{\cal P}(\R)$, we denote by $F_\gamma^{-1}(u)=\inf\{x\in\R:\gamma((-\infty,x])\ge u\},\; u\in (0,1)$ its quantile function. For $\eta,\gamma\in{\cal P}_1(\R)$, according to \cite[Theorem 3.A.5]{ShaShan}, we have
\begin{align}\label{carcvod1}
   \eta\le_{\rm cx}\gamma\Leftrightarrow \forall u\in[0,1],\;\int_0^uF_\eta^{-1}(v)dv\ge \int_0^uF_\gamma^{-1}(v)dv \mbox{ with equality when }u=1.
\end{align}
\begin{Lemma}\label{lemcocc1d}  Let $\eta,\mu,\nu\in{\cal P}_1(\R)$ and $\alpha\in (0,1)$. Then
\begin{align}
  {\rm COC}(\eta,\mu,\nu,\alpha)&\Leftrightarrow \eta\le_{\rm cx}\tilde\pi^\alpha_c\mbox{ where }\pi_c\mbox{ is the comonotonic coupling between $\mu$ and $\nu$}\notag\\&\Leftrightarrow \forall u\in[0,1],\;\int_0^uF_\eta^{-1}(v)dv\ge \int_0^u(\alpha F_\mu^{-1}(v)+(1-\alpha)F_\nu^{-1}(v))dv\notag\\&\phantom{\Leftrightarrow \forall u\in[0,1],\;}\mbox{ with equality when }u=1,\label{caracoc}\end{align}while\begin{align}
 \eta\le_{\rm cx}\alpha\mu+&(1-\alpha)\nu\notag\Leftrightarrow \forall u\in[0,1],\\&\int_0^uF_\eta^{-1}(v)dv\ge\inf_{w\in[0\vee (u+\alpha-1),u\wedge\alpha ]}\left(\alpha\int_0^{w/\alpha}F_\mu^{-1}(v)dv+(1-\alpha)\int_0^{(u-w)/(1-\alpha)}F_\nu^{-1}(v)dv\right)\notag\\&\mbox{ with equality when }u=1.\label{carac}
\end{align}
\end{Lemma}
\begin{Remark}\begin{itemize}
  \item 
    For the comonotonic coupling $\pi_c$ which is the image of the Lebesgue measure on $(0,1)$ by $(0,1)\ni v\mapsto(F_\mu^{-1}(v),F_\nu^{-1}(v))$, we have that \begin{equation}
      \forall u\in[0,1],\;\int_0^uF_{\tilde\pi^\alpha_c}^{-1}(v)dv=\int_0^u(\alpha F_\mu^{-1}(v)+(1-\alpha)F_\nu^{-1}(v))dv.\label{antquantcomo}
    \end{equation}
  \item When $u=1$, $[0\vee (u+\alpha-1),u\wedge\alpha ]=\{\alpha\}$ so that the right-hand side of the inequality in \eqref{carac} is equal to $\int_0^1(\alpha F_\mu^{-1}(v)+(1-\alpha)F_\nu^{-1}(v))dv$.
  \item For each $u\in[0,1]$, $\alpha u\in[0\vee (u+\alpha-1),u\wedge\alpha ]$ so that the right-hand side of the inequality in \eqref{carac} is not greater than that in \eqref{caracoc} and we recover \eqref{COCimpC}.
\end{itemize}
    
\end{Remark}
\noindent {\bf Proof of Lemma \ref{lemcocc1d}.}

In view of \eqref{carcvod1} and \eqref{antquantcomo}, to prove \eqref{caracoc} it is enough to check that
\begin{equation}
   \forall \pi\in\Pi(\mu,\nu),\;\forall u\in[0,1],\;\int_0^uF_{\tilde\pi^\alpha}^{-1}(v)dv\ge \int_0^u(\alpha F_\mu^{-1}(v)+(1-\alpha)F_\nu^{-1}(v))dv,\label{minobarpi}
\end{equation}
which we now do.

For $u\in[0,1]$, let ${\cal B}(u)=(1-u)\delta_0+u\delta_1$ denote the Bernoulli distribution with parameter $u$. According to \cite[Theorem 3.1.2]{raru}, we have
\begin{equation*}
   \forall \gamma\in{\cal P}_1(\R),\;\forall u\in[0,1],\;\int_0^{u}F_\gamma^{-1}(v)dv=\inf_{\beta\in\Pi(\gamma,{\cal B}(u))}\int_{\R^2}xy\beta(dx,dy),
 \end{equation*}
 by optimality of the anticomonotonic coupling given by the image of the Lebesgue measure on $(0,1)$ by $(0,1)\ni v\mapsto (F_\gamma^{-1}(v),F_{{\cal B}(u)}^{-1}(1-v))$. Under $\beta\in\Pi(\gamma,{\cal B}(u))$, $\gamma(dx)$ a.e., the conditional distribution of the second coordinate given that the first is equal to $x$ gives full weight to $\{0,1\}$. We deduce that $\beta\in\Pi(\gamma,{\cal B}(u))$ if and only if $\beta(dx,dy)=\gamma(dx){\cal B}(p(x))(dy)$ for some measurable function $p:\R\to[0,1]$ such that $\int_{\R}p(x)\gamma(dx)=u$. 
As a consequence,  \begin{equation}
   \forall u\in[0,1],\;\int_0^{u}F_\gamma^{-1}(v)dv=\inf_{p:\R\to[0,1]\mbox{ s.t. }\int_{\R}p(x)\gamma(dx)=u}\int_{\R}xp(x)\gamma(dx).\label{intquant}
 \end{equation}
 We deduce that for $\pi\in\Pi(\mu,\nu)$ and $u\in[0,1]$,
 \begin{align*}
   \int_0^{u}F_{\tilde\pi^\alpha}^{-1}(v)dv=\inf_{p:\R\to[0,1]\mbox{ s.t. }\int_{\R^2}p(\alpha x+(1-\alpha)y)\pi(dx,dy)=u}\int_{\R^2}(\alpha x+(1-\alpha)y)p(\alpha x+(1-\alpha)y)\pi(dx,dy).
 \end{align*}
 For $p$ as in the right-hand side, writing $\pi(dx,dy)=\mu(dx)\pi_x(dy)$ and setting $\rho(x)=\int_\R p(\alpha x+(1-\alpha)y)\pi_x(dy)$, we have that $\int_\R\rho(x)\mu(dx)=\int_{\R^2}p(\alpha x+(1-\alpha)y)\pi(dx,dy)=u$ and, by \eqref{intquant} applied with $\gamma=\mu$,
 $$\int_{\R^2}xp(\alpha x+(1-\alpha)y)\pi(dx,dy)=\int_{\R}x\rho(x)\mu(dx)\ge \int_0^{u}F_\mu^{-1}(v)dv.$$
 Dealing with the other contribution in a symmetric way, we conclude that \eqref{minobarpi} holds.

According to \eqref{carcvod1} and \eqref{intquant}, to prove \eqref{carac} it is enough to check that
for each $u\in [0,1]$, \begin{align}
   \int_0^uF_{\alpha\mu+(1-\alpha)\nu}^{-1}(v)dv=\inf_{w\in[0\vee (u+\alpha-1),u\wedge\alpha ]}\bigg(&\inf_{\rho:\R\to[0,1]\mbox{ s.t. }\int_{\R}\rho(x)\mu(dx)=\frac{w}{\alpha}}\alpha\int_\R x\rho(x)\mu(dx)\notag\\&+\inf_{q:\R\to[0,1]\mbox{ s.t. }\int_{\R}q(x)\nu(dx)=\frac{u-w}{1-\alpha}(1-\alpha)}\int_\R xq(x)\nu(dx)\bigg),\label{intquantcomb}
\end{align}
which we now do. Since for $p:\R\to[0,1]$ measurable such that $\int_{\R}p(x)(\alpha\mu+(1-\alpha)\nu)(dx)=u$, $\int_{\R}p(x)\mu(dx)=\frac w\alpha$ and $\int_{\R}p(x)\nu(dx)=\frac{u-w}{1-\alpha}$ for some $w\in[0\vee (u+\alpha-1),u\wedge\alpha ]$, \eqref{intquant} applied with $\gamma=\alpha\mu+(1-\alpha)\nu$ implies that the left-hand side in \eqref{intquantcomb} is not smaller than the right-hand side. On the other hand, for $w\in[0\vee (u+\alpha-1),u\wedge\alpha ]$, $\rho,q:\R\to[0,1]$ measurable such that $\int_{\R}\rho(x)\mu(dx)=\frac w\alpha$ and $\int_{\R}q(x)\nu(dx)=\frac{u-w}{1-\alpha}$, denoting by $f:\R\to[0,1]$ a density of $\alpha \mu$ with respect to $\alpha\mu+(1-\alpha)\nu$ so that $1-f$ is a density of $(1-\alpha)\nu$ with respect to $\alpha\mu+(1-\alpha)\nu$, we have that
\begin{align*}
  \alpha\int_\R x\rho(x)\mu(dx)+(1-\alpha)\int_\R xq(x)\nu(dx)=\int_\R x(f\rho+(1-f)q)(x)(\alpha\mu+(1-\alpha)\nu)(dx),\\
  u=\alpha\int_\R \rho(x)\mu(dx)+(1-\alpha)\int_\R q(x)\nu(dx)=\int_\R (f\rho+(1-f)q)(x)(\alpha\mu+(1-\alpha)\nu)(dx).
\end{align*}
Since $f\rho+(1-f)q$ is $[0,1]$-valued and measurable, with \eqref{intquant} applied with $\gamma=\alpha\mu+(1-\alpha)\nu$ again, we conclude that the converse inequality between both sides of \eqref{intquantcomb} holds.
\hfill$\Box$
\subsection{Convex order couplings for Gaussian distributions}
For Gaussian distributions, the implication \eqref{COCimpC} turns out to be an equivalence.
\begin{Proposition}
  Let ${\mathbb C},{\mathbb C_1},{\mathbb C_2}\in{\cal S}_+(d)$ and $\alpha\in[0,1]$. We have equivalence between
  \begin{itemize}
   \item [(a)]
   \begin{equation*}
      \exists \Theta\in\R^{d\times d},\;\left(\begin{array}{cc} {\mathbb C}_1 & \Theta \\ \Theta^* & {\mathbb C}_2\end{array}\right)\in{\cal S}_+(2d)\mbox{ and }{\mathbb C}\le \alpha^2 {\mathbb C}_1 +\alpha(1-\alpha)(\Theta+\Theta^*)+(1-\alpha)^2{\mathbb C}_2,
    \end{equation*}
  \item [(b)] $\forall x,y\in\R^d$, ${\rm COC}\left({\cal N}_d(\alpha x+(1-\alpha)y,{\mathbb C}),{\cal N}_d(x,{\mathbb C_1}),{\cal N}_d(y,{\mathbb C_2}),\alpha\right)$ holds,\item [(c)] $\forall x,y\in\R^d$, $\delta_{\alpha x+(1-\alpha)y}\otimes{\cal N}_d(\alpha x+(1-\alpha)y,{\mathbb C})\le_{\rm cx}\alpha\delta_x\otimes{\cal N}_d(x,{\mathbb C_1})+(1-\alpha)\delta_y\otimes{\cal N}_d(y,{\mathbb C_2})$,
  \item[(d)]$\forall x,y\in\R^d$, ${\cal N}_d(\alpha x+(1-\alpha)y,{\mathbb C})\le_{\rm cx} \alpha{\cal N}_d(x,{\mathbb C_1})+(1-\alpha){\cal N}_d(y,{\mathbb C_2})$,\item[(e)] ${\cal N}_d(0,{\mathbb C})\le_{\rm cx} \alpha{\cal N}_d(0,{\mathbb C_1})+(1-\alpha){\cal N}_d(0,{\mathbb C_2})$,
  \item[(f)] $\forall \zeta\in\R^d,\;\sqrt{\zeta^*{\mathbb C}\zeta}\le \alpha\sqrt{\zeta^*{\mathbb C}_1\zeta} +(1-\alpha)\sqrt{\zeta^*{\mathbb C}_2\zeta}.$
  \end{itemize}\label{propequivgauss}
\end{Proposition}
\noindent {\bf Proof.} The successive implications $(a)\Rightarrow (b)\Rightarrow (d)\Rightarrow (e)\Rightarrow (f)$ follow from \cite[Theorem 1]{JPJMVA} and its proof with the convex order coupling in $(b)$ equal to ${\cal N}_{2d}\left(\left(\begin{array}{c}x
   \\y
\end{array}\right),\left(\begin{array}{cc} {\mathbb C}_1 & \Theta \\ \Theta^* & {\mathbb C}_2\end{array}\right)\right)$. The implication $(b)\Rightarrow (c)$ follows from Lemma \ref{lemcoc2marg}, while we have $(c)\Rightarrow(d)$ since the convex order between two probability measures implies the convex order between their marginals. The implication $(f)\Rightarrow (a)$ is a consequence of the proof of \cite[Theorem 4]{gueant} which only states the equivalence of ${\cal N}_d(0,{\mathbb C})\le_{\rm cx} \alpha{\cal N}_d(0,{\mathbb C_1})+(1-\alpha){\cal N}_d(0,{\mathbb C_2})$ with the same inequality for the images by $\R^d\ni x\mapsto\zeta^*x$ for each $\zeta\in\R^d$ (one-dimensional projected convex order) which is easily shown to be equivalent to $(f)$ (see \cite{JPJMVA}). This proof ensures that when ${\mathbb C}, {\mathbb C}_1, {\mathbb C}_2$ are positive definite, there exists $O\in{\cal O}(d)$ such that, denoting for $i\in\{1,2\}$ by $L_i\in\R^{d\times d}$ the lower triangular matrix such that $L_iL_i^*= O{\mathbb C}^{-1/2}{\mathbb C}_i{\mathbb C}^{-1/2}O^*$ obtained by the Cholesky decomposition, $(a)$ holds with $\Theta={\mathbb C}^{1/2}O^*L_1L_2^*O{\mathbb C}^{1/2}$. When one of the matrices ${\mathbb C}, {\mathbb C}_1, {\mathbb C}_2$ is singular, then $(f)$ holds for ${\mathbb C}$, ${\mathbb C}_1$ and ${\mathbb C}_2$ replaced by ${\mathbb C}+\frac 1n I_d$, ${\mathbb C}_1+\frac 1n I_d$ and ${\mathbb C}_2+\frac 1 nI_d$ for $n\in\N^*$ as easily seen by squaring the inequalities. Hence, for each $n\in\N^*$, there exists $\Theta_n\in\R^{d\times d}$ such that $\left(\begin{array}{cc} {\mathbb C}_1 +\frac 1n I_d& \Theta_n \\ \Theta^*_n & {\mathbb C}_2+\frac 1n I_d\end{array}\right)\in{\cal S}_+(2d)$ and ${\mathbb C}+\frac 1n I_d\le \alpha^2 \left({\mathbb C}_1+\frac 1n I_d\right) +\alpha(1-\alpha)(\Theta_n+\Theta^*_n)+(1-\alpha)^2\left({\mathbb C}_2+\frac 1n I_d\right)$. Since $|(\Theta_n)_{ij}|\le \sqrt{(({\mathbb C}_1)_{ii}+\frac 1n)(({\mathbb C}_2)_{jj}+\frac 1n)}$ for $i,j\in\{1,\cdots,d\}$, we can extract from $(\Theta_n)_{n\in\N^*}$ a subsequence converging to some limit $\Theta\in\R^{d\times d}$. By taking the limit in the above properties of $\Theta_n$, we check that $\Theta$ satisfies the conditions in $(a)$.\hfill$\Box$
   
\section{Main results}\label{secmr}
We consider some measurable map $\R^d\ni x\mapsto a(x)\in{\cal S}_+(d)$ and introduce the associated martingale problem.
\begin{Definition}\label{defmp}\begin{itemize}
   \item Let $(t,x)\in [0,+\infty)\times\R^d$ and $(\xi_r)_{r\in[0,t]}$ denote the canonical process on $C([0,t],\R^d)$.
   A probability measure $P^x_{\le t}$ on $C([0,t],\R^d)$ is a solution to the martingale problem $(MP^x_{\le t})$ if $P^x_{\le t}(\xi_0=x)=1$ and 
  $$\left(M^\varphi_s=\varphi(\xi_s)-\varphi(x)-\frac 12\int_0^s{\rm tr}\left(a(\xi_r)\nabla^2\varphi(\xi_r)\right)
    dr\right)_{s\in[0,t]}$$
  is a $P^x_{\le t}$-martingale for each $\varphi:\R^d\to\R$ ${\cal C}^\infty$ with compact support.  
  \item We say that the martingale problem is well-posed if for each $(t,x)\in [0,+\infty)\times\R^d$,
$(MP^x_{\le t})$ admits a unique solution $P^x_{\le t}$.
\end{itemize}
\end{Definition}\begin{Remark}\label{remwpmp}
   
  Note that the martingale problem is well-posed
\begin{itemize}
\item when $\R^d\ni x\mapsto a(x)\in{\cal S}_+(d)$ is continuous, bounded and uniformly positive definite, according to \cite[Theorem 7.2.1]{StroVa},
\item when $d\in\{1,2\}$ and  $\R^d\ni x\mapsto a(x)\in{\cal S}_+(d)$ is measurable, bounded and locally uniformly positive definite, according to \cite[Exercises 7.3.2, 7.3.3 and 7.3.4]{StroVa}, 
  \item when $a(x)=\gamma(x)\sigma(x)\sigma^*(x)$ for some measurable function $\gamma:\R^d\to\R$ such that $0<\inf_{x\in\R^d}\gamma(x)\le \sup_{x\in\R^d}\gamma(x)<+\infty$ and some locally Lipschitz function $\sigma:\R^d\mapsto \R^{d\times q}$ with affine growth. This relies on the existence and trajectorial uniqueness for the stochastic differential equation with Brownian coefficient $\sigma$ (and no drift coefficient), the Yamada--Watanabe theorem and the method of time-change studied in \cite[Proposition 1.13 Chapter IX]{RevuzYor} which enables the introduction of the factor $\gamma$.
\end{itemize}
\end{Remark}
In our main result, we suppose that the martingale problem is well-posed. In its proof, we will need the convergence to  a solution to $(MP^x_{\le t})$ of the distribution of the Euler discretization of the associated stochastic differential equation as the number of time-steps goes to $\infty$.

To define the Euler scheme, we denote for $(z,s)\in\R^d\times (0,+\infty)$ by $\nu^z_s(d\tilde z)$ the Gaussian distribution ${\cal N}_d(z,a(z)s)$ involved in the scheme with step $s$. For $n\in \N^*$, $x\in\R^d$ and $t\in (0,+\infty)$, we consider the distribution $P^{n,x}_{\le t}$ of the linearly interpolated Euler scheme with $n$ steps on the time interval $[0,t]$ started from $x$. The probability measure  $P^{n,x}_{\le t}$ is the image of  $\delta_x(dx_0)\prod_{k=0}^{n-1}\nu^{x_k}_{t/n}(dx_{k+1})$ by the linear interpolation operator $i_n:(\R^d)^{n+1}\to C([0,t],\R^d)$ defined by
\begin{equation}
   i_n(x_{0:n})(s)=\frac{(k+1)t/n-s}{t/n}x_k+\frac{s-kt/n}{t/n}x_{k+1}\mbox{ for }s\in[kt/n,(k+1)t/n]\mbox{ with }k\in\{0,\cdots,n-1\}.\label{delfin}
\end{equation} \begin{Proposition}\label{prop:cveul}Let $(t,\rho)\in (0,+\infty)\times [1,+\infty)$ and assume \begin{description}
\item[(A1)] $\R^d\ni x\mapsto a(x)\in{\cal S}_+(d)$ is continuous and such that $\sup_{x\in\R^d}\frac{|a(x)|}{1+|x|^2}<\infty$.
\end{description}Then there exists some finite constant $C$ only depending on $d,t,\rho$ and $\sup_{x\in\R^d}\frac{|a(x)|}{1+|x|^2}$ such that $$\sup_{n\in\N^*,x\in\R^d}\frac 1{1+|x|^\rho}\int_{C([0,t],\R^d)}\sup_{s\in[0,t]}|\xi_s|^\rho P^{n,x}_{\le t}(d\xi)\le C.$$ Moreover, for each $x\in\R^d$, the sequence $(P^{n,x}_{\le t})_{n\in\N^*}$ is relatively compact in ${\cal P}_\rho(C([0,t],\R^d))$ and the limit of any converging subsequence solves the martingale problem $(MP^x_{\le t})$.
\end{Proposition}
The short proof of this rather standard result is postponed until after that of our main results in Section \ref{secproofs}.
We now introduce the assumption \begin{description}
   \item[(A2)] For each $(t,x)\in [0,+\infty)\times\R^d$, uniqueness holds for the martingale problem $(MP^x_{\le t})$,
\end{description}to state the following immediate consequence of  Proposition \ref{prop:cveul}.
\begin{Corollary}\label{corcveul}
  Under {\bf (A1)} and {\bf (A2)}, the martingale problem is well-posed. Moreover, for each $\rho\in[1,+\infty)$ and each $(t,x)\in (0,+\infty)\times\R^d$, ${\cal W}_\rho(P^{n,x}_{\le t},P^x_{\le t})\rightarrow 0$ as $n\to\infty$ where $P^x_{\le t}$ denotes the unique solution of $(MP^x_{\le t})$.
\end{Corollary}

When the martingale problem is well-posed and in particular under {\bf (A1)} and {\bf (A2)} in view of Corollary \ref{corcveul}, for $0\le s\le t$, $P^x_{\le s}$ is the image of $P^x_{\le t}$ by the restriction $C([0,t],\R^d)\ni (\xi_r)_{r\in[0,t]}\mapsto (\xi_r)_{r\in[0,s]}\in C([0,s],\R^d)$. For $t\in [0,+\infty)$, we denote by \begin{itemize}
\item $P^x_t$ the common image of $P^x_{\le T}$ by $C([0,T],\R^d)\ni (\xi_r)_{r\in[0,T]}\mapsto \xi_t\in \R^d$ for $T\in[t,+\infty)$,
\item $P^x_{0,t}$ the image of $P^x_{t}(d\tilde x)$ by $\R^d\ni\tilde x\mapsto (x,\tilde x)\in \R^d\times \R^d$.
\end{itemize} Note that $P^x_{0,t}$ is also the image of $P^x_{\le T}$ by $C([0,T],\R^d)\ni (\xi_r)_{r\in[0,T]}\mapsto (\xi_0,\xi_t)\in \R^d\times\R^d$ for any $T\in[t,+\infty)$.

\begin{Definition}\begin{itemize}
  \item For $t\ge 0$, when the martingale problem is well-posed and for each $x\in\R^d$, $P^x_{\le t}\in {\cal P}_1(C([0,t],\R^d))$, we say that $\R^d\ni x\mapsto P^x_{\le t}$ (resp. $x\mapsto P^x_t$) is coupling-convex if for each $(x,y,\alpha)\in\R^d\times \R^d\times[0,1]$, ${\rm COC}(P^{\alpha x+(1-\alpha)y}_{\le t},P^{x}_{\le t},P^{y}_{\le t},\alpha)$ (resp. ${\rm COC}(P^{\alpha x+(1-\alpha)y}_{t},P^{x}_{t},P^{y}_{t},\alpha)$) holds.
    \item In the same way, for $(n,t)\in\N^*\times \R_+$ such that for each $x\in\R^d$, $P^{n,x}_{\le t}\in {\cal P}_1(C([0,t],\R^d))$, we say that $\R^d\ni x\mapsto P^{n,x}_{\le t}$ is coupling-convex if for each $(x,y,\alpha)\in\R^d\times \R^d\times[0,1]$, ${\rm COC}(P^{n,\alpha x+(1-\alpha)y}_{\le t},P^{n,x}_{\le t},P^{n,y}_{\le t},\alpha)$ holds.
\end{itemize} 
\end{Definition}
When $x=y$ or $\alpha\in\{0,1\}$, ${\rm COC}(P^{\alpha x+(1-\alpha)y}_{t},P^{x}_{t},P^{y}_{t},\alpha)$ obviously holds (with the convex coupling equal to the image of $P^x_t$ by $\R^d\ni\tilde x\mapsto (\tilde x,\tilde x)\in\R^{2d}$ when $x=y$). With Lemma \ref{lemcoc2marg} and the equality $P^x_{0,t}=\delta_x\otimes P^x_t$ for $(t,x)\in\R_+\times\R^d$, we immediately deduce the following statement.
\begin{Proposition} \label{propeqcvccv}Assume that the martingale problem is well-posed and let $t\ge 0$ be such that $P^x_t\in{\cal P}_1(\R^d)$ for each $x\in\R^d$. Then
\begin{align*}
   &\R^d\ni x\mapsto P^x_t\mbox{ is coupling-convex}\\\Leftrightarrow &\R^d\ni x\mapsto P^x_{0,t}\mbox{ is convex when the codomain is endowed with the convex order.}
\end{align*}
\end{Proposition}\begin{Remark}\label{remeqcvccv}Note that one also has equivalence with convexity of $\R^d\ni x\mapsto \delta_x\otimes P^x_{0,t}$ (for $\varphi:(\R^d)^3\to\R$ convex, $\R^d\times\R^d\ni(z,w)\mapsto \varphi(z,z,w)$ keeps convex) and therefore with coupling-convexity of $\R^d\ni x\mapsto P^x_{0,t}$ (with obvious definition) by Lemma \ref{lemcoc2marg}. Note that when the martingale problem is well-posed and for each $x\in\R^d$, $P^x_{\le t}\in {\cal P}_1(C([0,t],\R^d))$, we also have equivalence of convexity and coupling-convexity of $\R^d\ni x\mapsto P^x_{\le t}$. Indeed, for $n\in\N^*$, we similarly derive equivalence of convexity and coupling-convexity of $\R^d\ni x\mapsto P^{x}_{(0:n)t/n}$ where $P^{x}_{(0:n)t/n}$ is the image of $P^x_{\le t}$ by $C([0,t],\R^d)\ni(\xi_s)_{s\in[0,t]}\mapsto (\xi_{kt/n})_{k=0:n}\in(\R^d)^{n+1}$. We can conclude by convergence of the image of $P^{x}_{(0:n)t/n}$ by the linear interpolation operator $i_n$ defined in \eqref{delfin} to $P^x_{\le t}$ in ${\cal P}_1(C([0,t],\R^d))$ (see Remark \ref{remcouple} below for more details). 
\end{Remark}

Let us state another equivalence which is an immediate consequence of Proposition \ref{propequivgauss}.
\begin{Proposition}
  \begin{align*}
   &\forall \zeta\in\R^d,\;\R^d\ni x\mapsto \sqrt{\zeta^*a(x)\zeta}\mbox{ is convex }\\\Leftrightarrow &\forall (x,y,\alpha)\in\R^d\times \R^d\times[0,1],\;\exists \Theta\in\R^{d\times d}\mbox{ such that }\\& \left(\begin{array}{cc} a(x) & \Theta \\ \Theta^* & a(y)\end{array}\right)\in{\cal S}_+(2d)\mbox{ and }a(\alpha x+(1-\alpha)y)\le \alpha^2 a(x)+\alpha(1-\alpha)(\Theta+\Theta^*)+(1-\alpha)^2a(y).
  \end{align*}\label{propeqmat}
\end{Proposition}
\begin{Remark}
   Note that the convexity of $\R^d\ni x\mapsto \sqrt{\zeta^*a(x)\zeta}$ for each $\zeta\in\R^d$ implies the continuity of $\R^d\ni x\mapsto \zeta^*a(x)\zeta$ for each $\zeta\in\R^d$ and thus the continuity of $\R^d\ni x\mapsto a(x)$ (first take $\zeta$ to be any vector in the canonical basis to get continuity of the diagonal entries and then take $\zeta$ to be the sum of two distinct such vectors to deduce continuity of the extra-diagonal entries). As a consequence, the second equivalent assertion in Proposition \ref{propeqmat} also implies continuity of $a$. 
\end{Remark}
We now are ready to state our main result.\begin{Theorem}Consider the following properties:
 \begin{description}\item $(i)$ $\forall \zeta\in\R^d$, $\R^d\ni x\mapsto \sqrt{\zeta^*a(x)\zeta}$ is convex.
 \item $(ii)$ $\forall (x,y,\alpha)\in\R^d\times \R^d\times[0,1]$, $\exists \Theta\in\R^{d\times d}$ such that
   \begin{equation}
      \left(\begin{array}{cc} a(x) & \Theta \\ \Theta^* & a(y)\end{array}\right)\in{\cal S}_+(2d)\mbox{ and }a(\alpha x+(1-\alpha)y)\le \alpha^2 a(x)+\alpha(1-\alpha)(\Theta+\Theta^*)+(1-\alpha)^2a(y),\label{defthet}
    \end{equation}

 \item $(iii)$ For each $(n,t)\in\N^*\times \R_+$, $\R^d\ni x\mapsto P^{n,x}_{\le t}$ is coupling-convex,
 \item $(iv)$ For each $t\ge 0$, $\R^d\ni x\mapsto P^x_{\le t}$ is coupling-convex, 
 \item $(v)$ For each $t\ge 0$, $\R^d\ni x\mapsto P^x_{t}$ is coupling-convex,
   
  \item $(vi)$ $\forall t\ge 0$, $\R^d\ni x\mapsto P^x_{0,t}$ is convex when the codomain is endowed with the convex order, 
  
\end{description}
  \begin{itemize}
  \item Under {\bf (A1)} and {\bf (A2)}, we have $(i)\Leftrightarrow (ii) \Leftrightarrow (iii)\Leftrightarrow (iv)\Leftrightarrow (v)\Leftrightarrow (vi)$ and all the above assertions are also equivalent to
  $$(vii)\;\forall t\ge 0,\R^d\ni x\mapsto P^x_{\le t}\mbox{ is convex when the codomain is endowed with the convex order,}$$

\item As soon as the martingale problem is well-posed and for each $(t,x)\in\R_+\times \R^d$, $P^x_{\le t}\in {\cal P}_1(C([0,t],\R^d))$, we have $(iv)\Rightarrow (v)\Leftrightarrow (vi)$.
 
\end{itemize}

 \label{mainthm}
\end{Theorem}

\begin{Remark}\label{remmt}
  \begin{itemize}\item We could remove the continuity of $a$ from {\bf (A1)} and add it to $(iii)$, $(iv)$, $(v)$, $(vi)$ and $(vii)$ to keep the equivalence between all the assertions.
\item The one-marginal coupling-convexity $(v)$ turns out to be equivalent to the trajectorial coupling-convexity $(iv)$ while the trajectorial convexity  $(vii)$ is equivalent to the two-marginals convexity $(vi)$ but not to the one-marginal convexity of $x\mapsto P^x_t$. Indeed the characterization of the latter by \cite{LM06} is weaker than $(i)$.
   \item Let us discuss the addition of a time-homogeneous affine drift to the dynamics i.e. the replacement $\frac 12{\rm tr}\left(a(\xi_r)\nabla^2\varphi(\xi_r)\right)$ by $\frac 12{\rm tr}\left(a(\xi_r)\nabla^2\varphi(\xi_r)\right)+(\beta+b\xi_r)\cdot\nabla\varphi(\xi_r)$ for $(\beta,b)\in\R^d\times \R^{d\times d}$ in Definition \ref{defmp}. When this more general martingale problem is well-posed in the sense of this definition, the equivalence $(i)\Leftrightarrow (ii) \Leftrightarrow (iii)\Leftrightarrow (iv)\Leftrightarrow (v)\Leftrightarrow (vi)\Leftrightarrow (vii)$ still holds under {\bf (A1)}. Note however that the first and second conditions in Remark \ref{remwpmp} ensure well-posedness of the new martingale problem only when $b=0$, while well-posedness holds under the third condition with $\gamma$ constant.
   \end{itemize}
 \end{Remark}
 Let us now compare the distributions of the solutions to two $d$-dimensional time-homogeneous driftless stochastic differential equations.
\begin{Theorem}\label{thmcpsol}
  Let $\hat a,\tilde a:\R^d\to{\cal S}_+(d)$ and $\hat \mu,\tilde\mu\in{\cal P}_1(\R^d)$ be such that $\hat \mu\le_{\rm cx}\tilde\mu$. Assume that {\bf (A1)} and {\bf (A2)} are satisfied with $a$ replaced by $\hat a$ and $\tilde a$. Let $\hat P^x_{\le t}$ and $\tilde P^x_{\le t}$ respectively denote the solutions on $[0,t]$ of the corresponding martingale problems started from $x$. Then the existence of $a:\R^d\to {\cal S}_+(d)$ such that
  $$\forall \zeta\in\R^d,\;\R^d\ni x\mapsto \sqrt{\zeta^*a(x)\zeta}\mbox{ is convex and }\forall x\in\R^d,\; \hat a(x)\le a(x)\le \tilde a(x)$$
  implies that for each $t\ge0$, $\int_{\R^d}\hat P^x_{\le t}\hat\mu(dx)\le_{\rm cx}\int_{\R^d}\tilde P^x_{\le t}\tilde \mu(dx)$.
\end{Theorem}

\section{Proofs of the main results}\label{secproofs}
 The proof of Theorem \ref{mainthm} relies on the next proposition for the implication
  $(vi)\Rightarrow (ii)$ and is postponed until after the one of this proposition.
\begin{Proposition}\label{propdec}Assume {\bf(A1)} and {\bf(A2)}. For each $x\in\R^d$, there exists a family $(p^x_t)_{t\ge 0}$ of probability measures on $\R^d$ such that $\lim_{t\to 0+}\frac 1 t\int_{\R^d}|\hat x|^2p^x_t(d\hat x)=0$ and a family $(\P^x_t)_{t\ge 0}$ of elements of $\Pi(\nu_t^x
  ,p^x_t)$ such that for each $t\ge 0$, $P^x_t$ is the image of $\P^x_t$ by $\R^d\times \R^d\ni (\tilde x,\hat x)\mapsto \tilde x+\hat x\in\R^d$. Moreover, denoting by ${\mathbb C}^x_t\in{\cal S}_+(d)$ the  covariance matrix of the probability distribution $P^x_t$, we have $\lim_{t\to 0+}\frac 1t {\mathbb C}^x_t=a(x)$ for each $x\in\R^d$.
\end{Proposition}
\noindent {\bf Proof of Proposition \ref{propdec}.}
By \cite[Proposition 4.11 Chapter 5]{KaShre}, there exists a weak solution to the stochastic differential equation
$$X^x_t=x+\int_0^ta^{1/2}(X^x_s)dW_s,\mbox{ where }(W_t)_{t\ge 0}\mbox{ is a $d$-dimensional Brownian motion},$$
such that for each $T\in[0,+\infty)$, $(X^x_r)_{r\in[0,T]}$ is distributed according to $P^x_{\le T}$. 
We define $\P^x_t$ as the distribution of $\left(x+a^{1/2}(x)W_t,\int_0^t\left(a^{1/2}(X^x_s)-a^{1/2}(x)\right)dW_s\right)$ so that its second marginal $p^x_t$ is the distribution of $\int_0^t\left(a^{1/2}(X^x_s)-a^{1/2}(x)\right)dW_s$. By It\^o's isometry, \begin{align}
             \int_{\R^d}|\hat x|^2p^x_t(d\hat x)=\int_0^t\E\left[{\rm tr}((a^{1/2}(X^x_s)-a^{1/2}(x))^2)\right]ds.\label{itois}\end{align}
By continuity of $a$, ${\rm tr}((a^{1/2}(X^x_s)-a^{1/2}(x))^2)$ converges a.s. to $0$ as $s\to 0+$. The quadratic growth assumption on $a$ in {\bf (A1)} combined with the Burkholder--Davis--Gundy inequality, a localization argument and Gronwall's lemma ensures that
$$\forall \rho,T\in[0,+\infty),\;\E\left[\sup_{t\in[0,T]}|X^x_t|^\rho\right]<\infty.$$
With the estimate$${\rm tr}((a^{1/2}(X^x_s)-a^{1/2}(x))^2)\le {\rm tr}(a(X^x_s)+a(x))\le C(1+|X^x_s|^2+|x|^2),$$
consequence of ${\rm tr}(AB)={\rm tr}(A^{1/2}BA^{1/2})\ge 0$ for $A,B\in{\cal S}_+(d)$ and of the growth assumption on $a$, we deduce uniform integrability of the family of random variables $\left({\rm tr}((a^{1/2}(X^x_s)-a^{1/2}(x))^2)\right)_{s\in[0,T]}$ for each $T\in (0,+\infty)$. Hence $\lim_{s\to 0+}\E\left[{\rm tr}((a^{1/2}(X^x_s)-a^{1/2}(x))^2)\right]=0$. With \eqref{itois}, we get that $\lim_{t\to 0+}\frac 1 t\int_{\R^d}|\hat x|^2p^x_t(d\hat x)=0$. Hence the covariance matrix $c_t$ of $p^x_t$ satisfies $\lim_{t\to 0+}\frac 1t c_t=0$. The covariance matrix of $\P^x_t$ reads $\left(\begin{array}{cc}a(x)t & \theta_t
   \\\theta^*_t &c_t
\end{array}\right)$ for some $\theta_t\in\R^{d\times d}$. Since for $i,j\in\{1,\cdots,d\}$, $\left(\begin{array}{cc}a_{ii}(x) & \frac 1t(\theta_t)_{ij}
   \\\frac 1t(\theta_t)_{ij} &\frac 1 t(c_t)_{jj}
\end{array}\right) \in{\cal S}_+(2)$ so that 
$(\frac 1t(\theta_t)_{ij})^2\le a_{ii}(x)(\frac 1t c_t)_{jj}$, we have $\lim_{t\to 0}\frac 1t\theta_t=0$. As ${\mathbb C}^x_t=a(x)t+\theta_t+\theta^*_t+c_t$, we conclude that $\lim_{t\to 0+}\frac 1t {\mathbb C}^x_t=a(x)$.\hfill$\Box$

\noindent {\bf Proof of Theorem \ref{mainthm}.} The second assertion holds since, under well-posedness of the martingale problem and finiteness of the first moment of its solutions, the implication $(iv)\Rightarrow (v)$ is clear and the equivalence $(v)\Leftrightarrow (vi)$ follows from Proposition \ref{propeqcvccv}. Since $(i)\Leftrightarrow (ii)$ by Proposition \ref{propeqmat}, to complete the proof of the first assertion, we will check that under {\bf (A1)} and {\bf (A2)},  $(ii)\Rightarrow (iii)$, $(iii)\Rightarrow (iv)$ and $(v)\Rightarrow (ii)$. Note that the equivalence with $(vii)$ follows from $(iv)\Rightarrow (vii)$, which is a consequence of \eqref{COCimpC}, and the obvious implication $(vii)\Rightarrow (vi)$.

{\bf Proof of $(ii)\Rightarrow (iii)$.}  For $(x,y,\beta)\in\R^d\times \R^d\times[0,1]$, the set \begin{align*}
   \bigg\{\Theta\in\R^{d\times d}:&\left(\begin{array}{cc} a(x) & \Theta \\ \Theta^* & a(y)\end{array}\right)\in{\cal S}_+(2d)\\&\mbox{ and }a(\beta x+(1-\beta)y)\le \beta^2 a(x)+\beta(1-\beta)(\Theta+\Theta^*)+(1-\beta)^2a(y)\bigg\}
\end{align*}
is a compact subset of $\R^{d\times d}$. By the Arsenin--Kunugui Theorem (see Theorem 18.18 \cite{Kechris}), we deduce that $(i)$ implies the existence of a measurable function $\Theta:\R^d\times\R^d\times[0,1]\ni(z,w,\beta)\mapsto \Theta(z,w,\beta)\in\R^{d\times d}$ such that for all $(z,w,\beta)\in\R^d\times\R^d\times[0,1]$, $\left(\begin{array}{cc} a(z) & \Theta(z,w,\beta) \\ \Theta^*(z,w,\beta) & a(w)\end{array}\right)\in{\cal S}_+(2d)$ and \begin{equation}
   a(\beta z+(1-\beta)w)\le \beta^2a(z)+\beta(1-\beta)(\Theta(z,w,\beta)+\Theta^*(z,w,\beta))+(1-\beta)^2a(w).\label{mizwb}
\end{equation}For $(z,w,\beta,s)\in\R^d\times \R^d\times[0,1]\times\R_+$, we denote by $N^\beta_s((z,w),(d\tilde z,d\tilde w))$ the Gaussian distribution ${\cal N}_{2d}\left(\left(\begin{array}{c}z\\  w\end{array}\right),\left(\begin{array}{cc}a(z) & \Theta(z,w,\beta)
   \\\Theta^*(z,w,\beta) & a(w)
\end{array}\right)s\right)$. Note that\begin{equation}
N^\beta_s((z,w),\cdot)\in\Pi(\nu_s^z,\nu_s^w).
\label{margN}
\end{equation}
The matrix inequality \eqref{mizwb} together with Lemma \ref{lemcvord} {\bf (2)} implies that                                                                                                                                                                                            \begin{equation}
  {\cal N}_d(\beta z+(1-\beta)w,a(\beta z+(1-\beta)w)s)=\nu_s^{\beta z+(1-\beta)w}\le_{\rm cx}((\tilde z,\tilde w)\mapsto \beta \tilde z+(1-\beta)\tilde w)\# N^\beta_s((z,w),(d\tilde z,d\tilde w))
\label{compmeas}.
 \end{equation}

Let now $(x,y,\alpha,t)\in\R^d\times \R^d\times[0,1]\times\R^*_+$ and $n\in\N^*$. 
We choose $((\tilde X^{n,x}_{kt/n},\tilde X^{n,y}_{kt/n})_{k=0:n})$ to be distributed according to \begin{equation}\delta_{(x,y)}(dx_0,dy_0)\prod_{k=0}^{n-1}N^\alpha_{t/n}((x_k,y_k),(dx_{k+1},dy_{k+1})).\label{loicoupl}
\end{equation}
By integration with respect to $(y_0,\cdots,y_n)\in(\R^d)^{n+1}$, we check using \eqref{margN} that $(\tilde X^{n,x}_{kt/n})_{k=0:n}$ is distributed according to $\delta_x(dx_0)\prod_{k=0}^{n-1}\nu_{t/n}^{x_k}(dx_{k+1})$ so that the law of $i_n\left((\tilde X^{n,x}_{kt/n})_{k=0:n}\right)$ is $P^{n,x}_{\le t}$. By symmetry, the laws of $(\tilde X^{n,y}_{kt/n})_{k=0:n}$ and $i_n\left((\tilde X^{n,y}_{kt/n})_{k=0:n}\right)$ are $\delta_y(dy_0)\prod_{k=0}^{n-1}\nu_{t/n}^{y_k}(dy_{k+1})$ and $P^{n,y}_{\le t}$. Let $Q^{n,x,y,\alpha}_{\le t}\in \Pi(P^{n,x}_{\le t},P^{n,y}_{\le t})$ denote the distribution of $\left(i_n\left((\tilde X^{n,x}_{kt/n})_{k=0:n}\right),i_n\left((\tilde X^{n,y}_{kt/n})_{k=0:n}\right)\right)$. We are going to prove by induction on $m\in\{0,\cdots,n\}$ that for $(X^{n,\alpha x+(1-\alpha)y}_{kt/n})_{k=0:n}$ distributed according to $\delta_{\alpha x+(1-\alpha)y}(dz_0)\prod_{k=0}^{n-1}\nu_{t/n}^{z_k}(dz_{k+1})$, we have \begin{equation*}
   (X^{n,\alpha x+(1-\alpha)y}_{kt/n})_{k=0:m}\le_{\rm cx}(\alpha\tilde X^{n,x}_{kt/n}+(1-\alpha)\tilde X^{n,y}_{kt/n})_{k=0:m}.\label{hyprecm}
\end{equation*} 

When $F:C([0,t],\R^d)\to\R$ is Lipschitz and convex, since $F\circ i_n:(\R^d)^{n+1}\to\R$ is convex as the composition of a convex function with a linear one, we deduce for $m=n$ that
\begin{align*}
   \int_{C([0,t],\R^d)}F((\tilde z_r)_{r\in [0,t]}&)P^{n,\alpha x+(1-\alpha)y}_{\le t}(d(\tilde z_r)_{r\in [0,t]})=\E[F(i_n((X^{n,\alpha x+(1-\alpha)y}_{kt/n})_{k=0:n}))]\\&\le \E[F(i_n((\alpha\tilde X^{n,x}_{kt/n}+(1-\alpha)\tilde X^{n,y}_{kt/n})_{k=0:n}))]\\&=\int_{C([0,t],\R^d)^2}F(\alpha (\tilde x_r)_{r\in [0,t]}+(1-\alpha) (\tilde y_r)_{r\in [0,t]})Q^{n,x,y,\alpha}_{\le t}(d(\tilde x_r)_{r\in [0,t]},d(\tilde y_r)_{r\in [0,t]}).
\end{align*}With Lemma \ref{lemcvord} {\bf (3)}, this implies ${\rm COC}(P^{n,\alpha x+(1-\alpha)y}_{\le t},P^{n, x}_{\le t},P^{n, y}_{\le t},\alpha)$.

The induction property 
is obvious for $m=0$. Let us suppose that it holds for $m\in\{0,\cdots,n-1\}$ and let $\varphi:(\R^d)^{m+2}\to\R$ be a convex function with at most affine growth. By the Markov property, we have $$\E\left[\varphi\left((X^{n,\alpha x+(1-\alpha)y}_{kt/n})_{k=0:m+1}\right)\right]=\E\left[\psi\left((X^{n,\alpha x+(1-\alpha)y}_{kt/n})_{k=0:m}\right)\right],$$ where $\psi(z_{0:m})=\int_{z_{m+1}\in\R^d}\varphi(z_{0:m},z_{m+1})\nu_{t/n}^{z_m}(dz_{m+1})$. For $(x_{0:m},y_{0:m},\beta)\in (\R^d)^{m+1}\times(\R^d)^{m+1}\times [0,1]$, we have, using the convexity of $\varphi$, the inequality \eqref{compmeas} and \eqref{margN},
\begin{align*}
   \psi(\beta x_{0:m}+&(1-\beta)y_{0:m})=\int_{v\in\R^d}\varphi(\beta x_{0:m}+(1-\beta)y_{0:m},v)\nu_{t/n}^{\beta x_m+(1-\beta) y_m}(dv)\\&\le \int_{(x_{m+1},y_{m+1})\in\R^d\times\R^d}\varphi(\beta x_{0:m+1}+(1-\beta)y_{0:m+1})N^\beta_{t/n}((x_m,y_m),(dx_{m+1},dy_{m+1}))\\&\le \int_{(x_{m+1},y_{m+1})\in\R^d\times\R^d}\left(\beta\varphi(x_{0:m+1})+(1-\beta)\varphi(y_{0:m+1})\right)N^\beta_{t/n}((x_m,y_m),(dx_{m+1},dy_{m+1}))\\&=\beta\psi(x_{0:m})+(1-\beta)\psi(y_{0:m}),
\end{align*}
so that $\psi$ is convex. Hence, using the induction property at rank $m$ for the first inequality then \eqref{compmeas} and the convexity of $\varphi$ for the second, we obtain 
\begin{align*}
  \E[&\varphi((X^{n,\alpha x+(1-\alpha)y}_{kt/n})_{k=0:m+1})]=\E[\psi((X^{n,\alpha x+(1-\alpha)y}_{kt/n})_{k=0:m})]\le \E[\psi((\alpha \tilde X^{n,x}_{kt/n}+(1-\alpha)\tilde X^{n,y}_{kt/n})_{k=0:m})]
  \\&=\E\left[\int_{\R^d}\varphi((\alpha \tilde X^{n,x}_{kt/n}+(1-\alpha)\tilde X^{n,y}_{kt/n})_{k=0:m},z_{m+1})\nu_{t/n}^{\alpha\tilde X^{n,x}_{mt/n }+(1-\alpha)\tilde X^{n,y}_{mt/n}}(dz_{m+1})\right]\\&\le \E\left[\int_{\R^d\times \R^d}\varphi((\alpha \tilde X^{n,x}_{kt/n}+(1-\alpha)\tilde X^{n,y}_{kt/n})_{k=0:m},\alpha x_{m+1}+(1-\alpha) y_{m+1})N^\alpha_{t/n}((\tilde X^{n,x}_{mt/n },\tilde X^{n,y}_{mt/n}),(dx_{m+1},dy_{m+1}))\right]\\&=\E[\varphi((\alpha \tilde X^{n,x}_{kt/n}+(1-\alpha)\tilde X^{n,y}_{kt/n})_{k=0:m+1})].
\end{align*}
Since $\varphi$ is an arbitrary convex function with affine growth, we deduce with Lemma \ref{lemcvord} {\bf (1)} that $(X^{n,\alpha x+(1-\alpha)y}_{kt/n})_{k=0:m+1}\le_{\rm cx}(\alpha \tilde X^{n,x}_{kt/n}+(1-\alpha)\tilde X^{n,y}_{kt/n})_{k=0:m+1}$.

{\bf Proof of $(iii)\Rightarrow (iv)$.}
Let $Q^{n,x,y,\alpha}_{\le t}\in\Pi(P^{n,x}_{\le t},P^{n,y}_{\le t})$ be such that its image $Q^{n,\alpha x+(1-\alpha)y}_{\le t}$ by $$C([0,t],\R^d)\times C([0,t],\R^d)\ni ((\tilde x_r)_{r\in [0,t]},(\tilde y_r)_{r\in [0,t]})\mapsto\alpha (\tilde x_r)_{r\in [0,t]}+(1-\alpha) (\tilde y_r)_{r\in [0,t]}\in C([0,t],\R^d)$$ satisfies  $P^{n,\alpha x+(1-\alpha)y}_{\le t}\le_{\rm cx}Q^{n,\alpha x+(1-\alpha)y}_{\le t}$. By Corollary \ref{corcveul}, $$\lim_{n\to\infty}{\cal W}_1(P^{n,x}_{\le t}, P^x_{\le t})+{\cal W}_1(P^{n,y}_{\le t}, P^y_{\le t})+{\cal W}_1(P^{n,\alpha x+(1-\alpha)y}_{\le t}, P^{\alpha x+(1-\alpha)y}_{\le t})=0.$$ Since $C([0,t],\R^d)$ is Polish, by the necessary condition in Prokhorov's theorem both sequences $(P^{n,x}_{\le t})_{n\in\N^*}$ and $(P^{n,y}_{\le t})_{n\in\N^*}$ are tight and so is $(Q^{n,x,y,\alpha}_{\le t})_{n\in\N^*}$. By the sufficient condition in Prokhorov's theorem, we can extract from $(Q^{n,x,y,\alpha}_{\le t})_{n\in\N^*}$ a subsequence still indexed by $n$ for notational simplicity such that $Q^{n,x,y,\alpha}_{\le t}$ converges weakly to some $Q^{x,y,\alpha}_{\le t}$. By continuity of the two projections from $C([0,t],\R^d)^2$ to  $C([0,t],\R^d)$, $Q^{x,y,\alpha}_{\le t}\in \Pi(P^{x}_{\le t},P^{y}_{\le t})$. Moreover, in view of \eqref{cvwr}, ${\cal W}_1(Q^{n,x,y,\alpha}_{\le t},Q^{x,y,\alpha}_{\le t})\to 0$ as $n\to\infty$. For $F:C([0,t],\R^d)\to\R$ Lipschitz and convex, 
\begin{align*}
   \int_{C([0,t],\R^d)}F((\tilde z_r&)_{r\in [0,t]})P^{n,\alpha x+(1-\alpha)y}_{\le t}(d(\tilde z_r)_{r\in [0,t]})\\&\le \int_{C([0,t],\R^d)^2}F(\alpha (\tilde x_r)_{r\in [0,t]}+(1-\alpha) (\tilde y_r)_{r\in [0,t]})Q^{n,x,y,\alpha}_{\le t}(d(\tilde x_r)_{r\in [0,t]},d(\tilde y_r)_{r\in [0,t]}).
\end{align*}
Since ${\cal W}_1(P^{n,\alpha x+(1-\alpha)y}_{\le t}, P^{\alpha x+(1-\alpha)y}_{\le t})+{\cal W}_1(Q^{n,x,y,\alpha}_{\le t},Q^{x,y,\alpha}_{\le t})\to 0$ as $n\to\infty$ and $F$ is Lipschitz continuous, the dual formulation \eqref{duw1} of ${\cal W}_1$ enables us to take the limit $n\to\infty$ in this inequality to conclude that it holds with $P^{n,\alpha x+(1-\alpha)y}_{\le t}$ and $Q^{n,x,y,\alpha}_{\le t}$ replaced by $P^{\alpha x+(1-\alpha)y}_{\le t}$ and $Q^{x,y,\alpha}_{\le t}$. We conclude with Lemma \ref{lemcvord} {\bf (3)}.

{\bf Proof of $(v)\Rightarrow (ii)$.} Assume $(v)$ and for   $(x,y,\alpha,t)\in\R^d\times \R^d\times[0,1]\times\R_+$, let $Q^{x,y,\alpha}_t\in\Pi(P^x_t,P^y_t)$ be such that its image $Q^{\alpha x+(1-\alpha)y}_t$ by $\R^d\times\R^d\ni (\tilde x,\tilde y)\mapsto\alpha \tilde x+(1-\alpha)\tilde y\in\R^d$ satisfies  $P^{\alpha x+(1-\alpha)y}_t\le_{\rm cx}Q^{\alpha x+(1-\alpha)y}_t$. Let ${\mathbb C}^{z}_t$ and $\widetilde{\mathbb C}^{\alpha x+(1-\alpha)y}_t$ denote the covariance matrices of the probability measures $P^z_t$ and $Q^{\alpha x+(1-\alpha)y}_t$. By Lemma \ref{lemcvord} {\bf (2)}, we have  \begin{align}
                                                                                                                                                                                                                                                                                                                                                                                                                                                                                                                                                                                                                              {\mathbb C}^{\alpha x+(1-\alpha)y}_t\le \widetilde{\mathbb C}^{\alpha x+(1-\alpha)y}_t.\label{compcoupl}\end{align}
Under the notation of Proposition \ref{propdec}, for $z\in\R^d$, we have $\P^z_t(d\tilde z,d\hat z)\delta_{\tilde z+\hat z}(d\check z)=P^z_t(d\check z)R^z_{t,\check z}(d\tilde z,d\hat z)$ for some probability kernel $\R^d\ni \check z\mapsto R^z_{t,\check z}\in{\cal P}(\R^d\times\R^d)$. Let $(\tilde X,\hat X,\tilde Y,\hat Y)$ be distributed according to $$\int_{(\check x,\check y)\in\R^d\times\R^d}Q^{x,y,\alpha}_t(d\check x,d\check y)R^{x}_{t,\check x}(d\tilde x,d\hat x)R^x_{t,\check y}(d\tilde y,d\hat y).$$ Then $\tilde X\sim{\cal N}_d(x,a(x)t)$, $\hat X\sim p^x_t$, $\tilde Y\sim {\cal N}_d(y,a(y)t)$, $\hat Y\sim p^y_t$ and $(\tilde X+\hat X,\tilde Y+\hat Y)\sim Q^{x,y,\alpha}_t$. Reasoning as in the end of the proof of Proposition \ref{propdec}, we check that 
${\mathbb C}\left(\begin{array}{c} 
  \tilde X+\hat X \\\tilde Y+\hat Y
\end{array}\right)=\left(\begin{array}{cc}a(x) & \Theta(t)\\
   \Theta^*(t)& a(y)
                                \end{array}\right)t+{o}(t)$ as $t\to 0+$ for some $\Theta(t)\in\R^{d\times d}$ such that the first term in the right-hand side is ${\mathbb C}\left(\begin{array}{c} \tilde X
   \\\tilde Y
\end{array}\right)$ and therefore the factor of $t$ in the right-hand side belongs to ${\cal S}_+(2d)$. Hence for $i,j\in\{1,\cdots,d\}$, $|\Theta_{ij}(t)|\le \sqrt{a_{ii}(x)a_{jj}(y)}$ and by compactness, we can find a sequence $(t_n)_{n}$ of positive times going to $0$ as $n\to\infty$ and such that $\Theta(t_n)$ converges to some limit $\Theta$. Of course, $\left(\begin{array}{cc}a(x) & \Theta\\
   \Theta^*& a(y)
                                \end{array}\right)\in{\cal S}_+(2d)$. Moreover, since $(\tilde X+\hat X,\tilde Y+\hat Y)\sim Q^{x,y,\alpha}_t$, $\alpha(\tilde X+\hat X)+(1-\alpha)(\tilde Y+\hat Y)\sim Q^{\alpha x+(1-\alpha)y}_t$ and the covariance matrix $\widetilde{\mathbb C}^{\alpha x+(1-\alpha)y}_t$ of this probability measure satisfies \begin{align*}
               \widetilde{\mathbb C}^{\alpha x+(1-\alpha)y}_t&=\left(\begin{array}{cc}\alpha I_d &
   (1-\alpha)I_d\end{array}\right){\mathbb C}\left(\begin{array}{c} 
  \tilde X+\hat X \\\tilde Y+\hat Y
\end{array}\right)\left(\begin{array}{c}\alpha I_d \\
   (1-\alpha)I_d\end{array}\right)\\&=\left(\alpha^2a(x)+\alpha(1-\alpha)(\Theta(t)+\Theta^*(t))+(1-\alpha)^2a(y)\right)t+o(t).                  
                                \end{align*}
With Proposition \ref{propdec} then \eqref{compcoupl}, we deduce that
       for $t>0$                       \begin{align*}
           a(\alpha x+(1-\alpha)y)+o(1)&=\frac 1 t{\mathbb C}^{\alpha x+(1-\alpha)y}_t\\&\le\frac 1 t  \widetilde{\mathbb C}^{\alpha x+(1-\alpha)y}_t=\alpha^2a(x)+\alpha(1-\alpha)(\Theta(t)+\Theta^*(t))+(1-\alpha)^2a(y)        +o(1).          
                              \end{align*}
Writing this inequality for $t=t_n$ and letting $n\to\infty$, we conclude that $
a(\alpha x+(1-\alpha)y)\le \alpha^2a(x)+\alpha(1-\alpha)(\Theta+\Theta^*)+(1-\alpha)^2a(y)$.

\hfill$\Box$
\begin{Remark}\label{remcnlm} Let us give a direct proof of $(vi)\Rightarrow (i)$ under {\bf (A1)} and {\bf (A2)}.
 By Proposition \ref{propdec}, for $z,\zeta\in\R^d$, \begin{align*}
   &\frac 1{\sqrt{t}}\int_{\R^d}|\zeta^*(w-z)|P_t^{z}(dw)=\frac 1{\sqrt{t}}\int_{\R^d\times \R^d}|\zeta^*(\tilde z+\hat z-z)|\P^z_t(d\tilde z,d\hat z)\\&\frac 1{\sqrt{t}}\int_{\R^d\times \R^d}|\zeta^*(\tilde z-z)|\P^z_t(d\tilde z,d\hat z)=\frac 1{\sqrt{t}}\int_{\R^d}|\zeta^*(\tilde z-z)|\nu^z_t(d\tilde z)=\sqrt{2\zeta^*a(z)\zeta /\pi}\\&\frac 1{\sqrt{t}}\int_{\R^d\times \R^d}|\zeta^*\hat z|\P^z_t(d\tilde z,d\hat z)=\frac 1{\sqrt{t}}\int_{\R^d}|\zeta^*\hat z|p^z_t(d\hat z)
                                                .
                                              \end{align*} Hence, by the triangle and Cauchy--Schwarz inequalities,
\begin{align*}
   \left|\frac 1{\sqrt{t}}\int_{\R^d}|\zeta^*(w-z)|P_t^{z}(dw)-\sqrt{2\zeta^*a(z)\zeta /\pi}\right|&=\frac 1{\sqrt{t}}\left|\int_{\R^d\times \R^d}(|\zeta^*(\tilde z+\hat z-z)|-|\zeta^*(\tilde z-z)|)\P^z_t(d\tilde z,d\hat z)\right|\\&\le\frac 1{\sqrt{t}}\int_{\R^d}|\zeta^*\hat z|p^z_t(d\hat z)\le |\zeta|\sqrt{\frac 1 t\int_{\R^d}|\hat z|^2p^z_t(d\hat z)},
\end{align*}
where the right-hand side converges to $0$ as $t\to 0+$, still by Proposition \ref{propdec}.

Let $(x,y,\alpha)\in\R^d\times \R^d\times[0,1]$. Under $(vi)$, we have $P_{0,t}^{\alpha x+(1-\alpha)y}\le_{\rm cx}\alpha P_{0,t}^x+(1-\alpha)P_{0,t}^y$. By convexity of $\R^d\times\R^d\ni(z,w)\mapsto \frac 1{\sqrt{t}}|\zeta^*(w-z)|$, we deduce that
\begin{align}
   \frac 1{\sqrt{t}}\int_{\R^d}|\zeta^*(w-(\alpha x+(1-\alpha)y))|P_t^{\alpha x+(1-\alpha)y}(dw)\le &\frac \alpha{\sqrt{t}} \int_{\R^d}|\zeta^*(w-x)|P_t^{x}(dw)\notag\\&+\frac{1-\alpha}{\sqrt{t}}\int_{\R^d}|\zeta^*(w-y)|P_t^{y}(dw).\label{inegzet}
\end{align}
We conclude that $\sqrt{\zeta^*a(\alpha x+(1-\alpha)y)\zeta}\le \alpha\sqrt{\zeta^*a(x)\zeta}+(1-\alpha)\sqrt{\zeta^*a(y)\zeta}$ by taking the limit $t\to 0+$ in this inequality multiplied by $\sqrt{\pi/2}$.

One easily checks in the same way that 
   $P_{t}^{\alpha x+(1-\alpha)y}\le_{\rm cx}\alpha P_{t}^x+(1-\alpha)P_{t}^y$ for each $t\ge 0$ implies $$\forall \zeta\in\R^d\mbox{ such that }\zeta^*(y-x)=0,\;\sqrt{\zeta^*a(\alpha x+(1-\alpha)y)\zeta}\le \alpha\sqrt{\zeta^*a(x)\zeta}+(1-\alpha)\sqrt{\zeta^*a(y)\zeta}.$$
   Indeed, when $\zeta^*(y-x)=0$, $\zeta^*(\alpha x+(1-\alpha)y)=\zeta^*x=\zeta^*y$  so that \eqref{inegzet} still holds by convexity of $\R^d\ni w\mapsto \frac 1{\sqrt{t}}|\zeta^*(w-x)|$.

Note that, with the notations of the proof of $(v)\Rightarrow (ii)$ in Theorem \ref{mainthm},
\begin{align}
   P_{0,t}^{\alpha x+(1-\alpha)y}\le_{\rm cx}\alpha P_{0,t}^x+(1-\alpha)P_{0,t}^y\Rightarrow {\mathbb C}^{\alpha x+(1-\alpha)y}_t\le \alpha {\mathbb C}^x_t+(1-\alpha){\mathbb C}^y_t\label{compsanscoupl}.
\end{align}For the convex function $\R^d\times\R^d\ni(z,w)\mapsto (\zeta^*(w-z))^2$ where $\zeta\in\R^d$, the convex order inequality indeed implies that $\zeta^*{\mathbb C}^{\alpha x+(1-\alpha)y}_t\zeta\le \alpha \zeta^*{\mathbb C}^x_t\zeta+(1-\alpha)\zeta^*{\mathbb C}^y_t\zeta$.
The gap between the right-hand sides of the covariance inequalities in \eqref{compsanscoupl} and \eqref{compcoupl} reads $$\alpha {\mathbb C}^x_t+(1-\alpha){\mathbb C}^y_t-\widetilde{\mathbb C}^{\alpha x+(1-\alpha)y}_t=\alpha(1-\alpha)\widehat {\mathbb C}^{x,y,\alpha}_t,$$ where $\widehat{\mathbb C}^{x,y,\alpha}_t\in{\cal S}_+(d)$ is the covariance matrix of the image of the probability measure $Q^{x,y,\alpha}_t(d\tilde x, d\tilde y)$ by $\R^d\times\R^d\ni (\tilde x,\tilde y)\mapsto \tilde y-\tilde x\in\R^d$
. This symmetric positive semi-definite gap explains why we are not able to directly recover $(ii)$ from $(vi)$. Finally,

\begin{align*}
&P^{\alpha x+(1-\alpha)y}_t\le_{\rm cx}\alpha P^x_t+(1-\alpha)P^y_t\Rightarrow {\mathbb C}^{\alpha x+(1-\alpha)y}_t\le \alpha {\mathbb C}^x_t+(1-\alpha){\mathbb C}^y_t+\alpha(1-\alpha)(x-y)(x-y)^*
\end{align*}
where the complementary contribution $\alpha(1-\alpha)(x-y)(x-y)^*$ is also symmetric positive semi-definite.\end{Remark}

\begin{Remark}\label{remcouple}Note that under {\bf (A1)} and {\bf (A2)}, one can prove directly that $(v)\Rightarrow (iv)$ without going through the other assertions. Indeed, for $n\in\N^*$and $(z,w,\beta)\in\R^d\times\R^d\times[0,1]$, \begin{align*}
   \bigg\{&Q\in\Pi(P^z_{t/n},P^w_{t/n})\mbox{ such that }P^{\beta z+(1-\beta)w}_{t/n}\mbox{ is smaller in the convex order than the image of $Q$ by}\\&\R^d\times\R^d\ni(\tilde x,\tilde y)\mapsto\beta \tilde x +(1-\beta)\tilde y\in\R^d\bigg\}
 \end{align*}
is a compact subset of ${\cal P}(\R^d\times\R^d)$ endowed with the weak convergence topology. By the Arsenin--Kunugui theorem (see \cite[Theorem 18.18]{Kechris}), we deduce that under $(v)$ there exists a measurable function $\R^d\times\R^d\times[0,1]\ni(z,w,\beta)\mapsto Q^{z,w,\beta}_{t/n}$ such that $Q^{z,w,\beta}_{t/n}$ belongs to this set. For $(x,y,\alpha)\in\R^d\times\R^d\times[0,1]$, choosing $((\tilde X^{n,x}_{kt/n},\tilde X^{n,y}_{kt/n})_{k=0:n})$ distributed according to \begin{equation}\delta_{(x,y)}(dx_0,dy_0)\prod_{k=0}^{n-1}Q^{x_k,y_k,\alpha}_{t/n}(dx_{k+1},dy_{k+1})
\end{equation}
instead of \eqref{loicoupl}, one can conclude with a similar reasoning as in the above proof of $(ii)\Rightarrow (iii)$ in Theorem \ref{mainthm} that for $(X^{n,\alpha x+(1-\alpha)y}_{kt/n})_{k=0:n}$ distributed according to $\delta_{\alpha x+(1-\alpha)y}(dz_0)\prod_{k=0}^{n-1}P_{t/n}^{z_k}(dz_{k+1})$\begin{equation*}
   (X^{n,\alpha x+(1-\alpha)y}_{kt/n})_{k=0:n}\le_{\rm cx}(\alpha\tilde X^{n,x}_{kt/n}+(1-\alpha)\tilde X^{n,y}_{kt/n})_{k=0:n}.
 \end{equation*}
 Note that $(\tilde X^{n,x}_{kt/n})_{k=0:n}$ and $(\tilde X^{n,y}_{kt/n})_{k=0:n}$ are respectively distributed according to $\delta_x(dx_0)\prod_{k=0}^{n-1}P_{t/n}^{x_k}(dx_{k+1})$ and $\delta_y(dy_0)\prod_{k=0}^{n-1}P_{t/n}^{y_k}(dy_{k+1})$.
By the Markov property 
 , for $z\in \R^d$, $\delta_{z}(dz_0)\prod_{k=0}^{n-1}P_{t/n}^{z_k}(dz_{k+1})$ is the image of $P^z_{\le t}$ by $C([0,t],\R^d)\ni \xi\mapsto (\xi_{kt/n})_{k=0:n}$. For $\xi\in C([0,t],\R^d)$ and $i_n$ defined in \eqref{delfin}, $$\sup_{s\in[0,t]}|\xi_s-i_n(\xi)(s)|\le\sup_{0\le r\le s\le (r+t/n)\wedge t}|\xi_r-\xi_s| \le 2\sup_{s\in[0,t]}|\xi_s|$$ where the modulus of continuity $\sup_{0\le r\le s\le (r+t/n)\wedge t}|\xi_r-\xi_s|$ goes to $0$ as $n\to\infty$ and the right-hand side is integrable with respect to $P^z_{\le t}$ by the quadratic growth assumption on $a$ in {\bf (A1)}. As a consequence, with $i_n\#\delta_{z}(dx_0)\prod_{k=0}^{n-1}P^{x_{k}}_{t/n}(dx_{k+1})$ denoting the image of the probability measure $\delta_{z}(dx_0)\prod_{k=0}^{n-1}P^{x_{k}}_{t/n}(dx_{k+1})$ by the linear interpolation operator $i_n$, 
\begin{equation}
   \forall z\in\R^d,\;\lim_{n\to\infty}{\cal W}_1\left(i_n\#\delta_{z}(dx_0)\prod_{k=0}^{n-1}P^{x_{k}}_{t/n}(dx_{k+1}),P^{z}_{\le t}\right)=0.\label{cvpz}
 \end{equation}
 We deduce $(iv)$ as in the proof of $(iii)\Rightarrow (iv)$ in Theorem \ref{mainthm}.\end{Remark}
\begin{Remark}
  To prove directly that $(vi)\Rightarrow (vii)$ under {\bf (A1)} and {\bf (A2)}, one can combine \eqref{cvpz} with the fact that for  $m\in\{0,\cdots,n\}$ and $(x,y,\alpha)\in\R^d\times\R^d\times[0,1]$, $$\delta_{\alpha x+(1-\alpha)y}(dx_0)\prod_{k=0}^{m-1}P^{x_{k}}_{t/n}(dx_{k+1})\le_{\rm cx}\alpha \delta_{x}(dx_0)\prod_{k=0}^{m-1}P^{x_{k}}_{t/n}(dx_{k+1})+(1-\alpha) \delta_{y}(dx_0)\prod_{k=0}^{m-1}P^{x_{k}}_{t/n}(dx_{k+1}).$$
  This inequality can be proved by induction on $m\in\{0,\cdots,n\}$. The hypothesis clearly holds for $m=0$. For a convex function $\varphi:(\R^d)^{m+2}\to\R$, Lemma \ref{lemcoc2marg} and the equality $P^z_{0,t/n}(dx_m,dx_{m+1})=\delta_z(dx_m)P^z_{t/n}(dx_{m+1})$ ensure that under $(v)$, the function
  \begin{align*}
   (\R^d)^{m+1}\ni x_{0:m}\mapsto &\int_{x_{m+1}\in\R^d}\varphi(x_{0:m+1})P^{x_{m}}_{t/n}(dx_{m+1})\\&=\int_{y_{0:m+1}\in(\R^d)^{m+2}}\varphi(y_{0:m+1})\delta_{x_{0:m}}(dy_{0:m})P^{x_{m}}_{t/n}(dy_{m+1})
  \end{align*}
  is convex so that the induction hypothesis at rank $m+1$ follows from the hypothesis at rank $m$.

\end{Remark}

\begin{Remark} If \eqref{defthet} holds with $\Theta=\Theta(x,y,\alpha)$ for some measurable function $\R^d\times\R^d\times[0,1]\ni(x,y,\alpha)\mapsto \Theta(x,y,\alpha)\in\R^{d\times d}$ continuous in its first two variables, then under {\bf (A1)} and {\bf (A2)}, $\R^d\times\R^d\ni(x,y)\mapsto A(x,y,\alpha)=\left(\begin{array}{cc} a(x) & \Theta(x,y,\alpha) \\ \Theta^*(x,y,\alpha) & a(y)\end{array}\right)\in\R^{2d\times 2d}$ is continuous and $(iv)$ holds with the coupling measure $P^{x,y,\alpha}_{\le t}\in\Pi(P^{x}_{\le t},P^{y}_{\le t})$ (the image of which by $C([0,t],\R^d)\times C([0,t],\R^d)\ni (\tilde x,\tilde y)\mapsto \alpha \tilde x+(1-\alpha)\tilde y$ dominates $P^{\alpha x+(1-\alpha)y}_{\le t}$ in the convex order) solving the martingale problem: $P^{x,y,\alpha}_{\le t}(\upsilon_0=(x,y))=1$ and
  $$\forall \varphi:\R^{2d}\to\R\;{\cal C}^\infty\mbox{ with compact support},\;\varphi(\upsilon_s)-\frac 12\int_0^s{\rm tr}(A(\upsilon_r,\alpha)\nabla^2\varphi(\upsilon_r))dr\mbox{ is a $P^{x,y,\alpha}_{\le t}$-martingale},$$
  where $(\upsilon_t)_{t\ge 0}$ denotes the canonical process in ${\cal C}([0,t],\R^{2d})$. 
\end{Remark}\noindent {\bf Proof of Theorem \ref{thmcpsol}.}
For $\nu^z_t={\cal N}_d(z,a(z)t)$, $\hat\nu^z_t={\cal N}_d(z,\hat a(z)t)$ and $\tilde\nu^z_t={\cal N}_d(z,\tilde a(z)t)$ where $(t,z)\in\R_+\times\R^d$, we have $
\hat\nu^z_t\le_{\rm cx}\nu^z_t\le_{\rm cx}\tilde \nu^z_t$. For $n\in\N^*$, let us check by induction on $m\in\{0,\cdots,n\}$ that $\hat\mu(dx_0)\prod_{k=0}^{m-1}\hat\nu^{x_k}_{t/n}(dx_{k+1})\le_{\rm cx} \tilde\mu(dx_0)\prod_{k=0}^{m-1}\tilde\nu^{x_k}_{t/n}(dx_{k+1})$. It is easy to check that the two probability measures in this inequality belong to ${\cal P}_1((\R^d)^{m+1})$ by performing some induction on $m$ using $\hat\mu,\tilde\mu\in{\cal P}_1(\R^d)$ and the quadratic growth property of $\hat a,\tilde a$ (or using the more refined estimate in Proposition \ref{prop:cveul}). The induction hypothesis holds for $m=0$. Assume that it holds at rank $m$ and let $\varphi:(\R^d)^{m+2}\to\R$ be convex. We set $\psi(x_{0:m})=\int_{x_{m+1}\in\R^d}\varphi(x_{0:m+1})\nu^{x_m}_{t/n}(dx_{m+1})$ and define $\hat\psi$ and $\tilde \psi$ in the same way by replacing $\nu^{x_m}_{t/n}$ by $\hat \nu^{x_m}_{t/n}$ and $\tilde \nu^{x_m}_{t/n}$ respectively. The previous convex order comparison yields $\hat\psi\le\psi\le \tilde \psi$. 
Moreover, by the implication $(i)\Rightarrow (ii)$ and the proof of $(ii)\Rightarrow(iii)$ in Theorem \ref{mainthm}, the function $\psi$ is convex. Using the definitions of $\hat\psi$ and $\tilde \psi$ for the equalities, the inequalities $\hat\psi\le\psi$ and $\psi\le\tilde\psi$ for the first and third inequalities and the convexity of $\psi$ together with the induction hypothesis at rank $m$ for the second inequality, we obtain that
\begin{align*}
   \int_{(\R^d)^{m+2}}\varphi(x_{0:m+1})\hat\mu(dx_0)\prod_{k=0}^m\hat\nu^{x_k}_{t/n}(dx_{k+1})&=
\int_{(\R^d)^{m+1}}\hat\psi(x_{0:m})\hat\mu(dx_0)\prod_{k=0}^{m-1}\hat\nu^{x_k}_{t/n}(dx_{k+1})\\&\le\int_{(\R^d)^{m+1}}\psi(x_{0:m})\hat\mu(dx_0)\prod_{k=0}^{m-1}\hat\nu^{x_k}_{t/n}(dx_{k+1})\\&\le\int_{(\R^d)^{m+1}}\psi(x_{0:m})\tilde\mu(dx_0)\prod_{k=0}^{m-1}\tilde\nu^{x_k}_{t/n}(dx_{k+1})\\&\le\int_{(\R^d)^{m+1}}\tilde\psi(x_{0:m})\tilde\mu(dx_0)\prod_{k=0}^{m-1}\tilde\nu^{x_k}_{t/n}(dx_{k+1})\\&= \int_{(\R^d)^{m+2}}\varphi(x_{0:m+1})\tilde\mu(dx_0)\prod_{k=0}^m\tilde\nu^{x_k}_{t/n}(dx_{k+1}).\end{align*}
We conclude that the induction hypothesis holds at rank $m+1$. For $x\in\R^d$, let $\hat P^{n,x}_{\le t}$ and $\tilde P^{n,x}_{\le t}$ denote the respective images of $\delta_x(dx_0)\prod_{k=0}^{n-1}\hat\nu^{x_k}_{t/n}(dx_{k+1})$ and $\delta_x(dx_0)\prod_{k=0}^{n-1}\tilde\nu^{x_k}_{t/n}(dx_{k+1})$ by the linear interpolation operator $i_n$ defined in \eqref{delfin}. For $m=n-1$ and $\varphi=F\circ i_n$ where $F:C([0,t],\R^d)\to\R$ is convex, the inequality between the extreme sides of the last equation reads
$$\int_{\R^d\times C([0,t],\R^d)}F(\xi)\hat P^{n,x}_{\le t}(d\xi)\hat\mu(dx)\le \int_{\R^d\times C([0,t],\R^d)}F(\xi)\tilde P^{n,x}_{\le t}(d\xi)\tilde\mu(dx).$$
We have ${\cal W}_1\left(\int_{\R^d}\hat P^{n,x}_{\le t}\hat\mu(dx),\int_{\R^d}\hat P^{x}_{\le t}\hat\mu(dx)\right)\le \int_{\R^d}{\cal W}_1(\hat P^{n,x}_{\le t},\hat P^{x}_{\le t})\hat\mu(dx)$, where Corollary \ref{corcveul} combined with the estimate in Proposition \ref{prop:cveul} and the dominated convergence theorem ensure that the right-hand side goes to $0$ as $n\to\infty$. In the same way, $\lim_{n\to\infty}{\cal W}_1\left(\int_{\R^d}\tilde P^{n,x}_{\le t}\tilde\mu(dx),\int_{\R^d}\tilde P^{x}_{\le t}\tilde\mu(dx)\right)=0$. When $F$ is moreover Lipschitz, we can thus take the limit $n\to\infty$ in the previous inequality to obtain that
$$\int_{\R^d\times C([0,t],\R^d)}F(\xi)\hat\mu(dx)\hat P^{x}_{\le t}(d\xi)\le \int_{\R^d\times C([0,t],\R^d)}F(\xi)\tilde\mu(dx)\tilde P^{x}_{\le t}(d\xi).$$
We conclude with Lemma \ref{lemcvord} {\bf (3)}.
\hfill$\Box$

\noindent {\bf Proof of Proposition \ref{prop:cveul}.} Under {\bf (A1)}, $\R^d\ni x\mapsto a^{1/2}(x)\in\R^{d\times d}$ is continuous and such that $\sup_{x\in\R^d}\frac{|a^{1/2}(x)|}{1+|x|}<\infty$. Let $(W_s)_{s\ge 0}$ be a $d$-dimensional Brownian motion. For $x\in\R^d$, let $(X^{n,x}_s)_{s\in[0,t]}$ be initialized by  $X^{n,x}_0=x$ and defined inductively for $k\in\{0,\cdots,n-1\}$ by
\begin{equation}
   \forall s\in [kt/n,(k+1)t/n],\;X^{n,x}_s=X^{n,x}_{kt/n}+a^{1/2}(X^{n,x}_{kt/n})(W_s-W_{kt/n}).\label{euler}
 \end{equation}
 Since $X^{n,x}_{(k+1)t/n}=X^{n,x}_{kt/n}+a^{1/2}(X^{n,x}_{kt/n})(W_{(k+1)t/n}-W_{kt/n})$ with $(W_{(k+1)t/n}-W_{kt/n})$ independent of $(X^{n,x}_{\ell t/n})_{\ell=0:k}$, $(X^{n,x}_{kt/n})_{k=0:n}$ is distributed according to $\delta_x(dx_0)\prod_{k=0}^{n-1}\nu_{t/n}^{x_k}(dx_{k+1})$ so that $i_n\left((X^{n,x}_{kt/n})_{k=0:n}\right)$ is distributed according to $P^{n,x}_{\le t}$. Moreover, using the independence between $X^{n,x}_{kt/n}$ and $(W_{(k+1)t/n}-W_{kt/n})$, one easily checks by induction on $k$ that for each $\rho\in [1,+\infty)$, $\max_{k=0:n}\E[|X^{n,x}_{kt/n}|^\rho]<\infty$.
For $\rho\in[2,+\infty)$ and $s\in[0,t]$, since $$\sup_{r\in[0,s]}|X^{n,x}_r|^\rho \le 2^{\rho-1}\left(|x|^\rho+\sup_{r\in[0,s]}\left|\int_0^r
  a^{1/2}(X^{n,x}_{\lfloor n u/t\rfloor t/n})dW_u\right|^\rho\right),$$ by the Burkholder-Davis-Gundy inequality and the growth assumption on $a$,
$$\E\left[\sup_{r\in[0,s]}|X^{n,x}_r|^\rho\right]\le C\left(1+|x|^\rho+\int_0^s\E\left[\sup_{r\in[0,u]}|X^{n,x}_r|^\rho\right]du
\right),$$
for some finite constant $C$ only depending on $d,t,\rho$ and $\sup_{x\in\R^d}\frac{|a(x)|}{1+|x|^2}$. With Gronwall's lemma and the inequality $\sup_{s\in[0,t]}|i_n(x_{0:n})(s)|\le \max_{k=0:n}|x_k|$, we conclude that for each $\rho\in[1,+\infty)$,
\begin{align}
   \sup_{n\in\N^*,x\in\R^d}\frac{1}{1+|x|^\rho}\int_{C([0,t],\R^d)}\sup_{s\in[0,t]}|\xi_s|^\rho P^{n,x}_{\le t}(d\xi)\le\sup_{n\in\N^*,x\in\R^d}\frac{1}{1+|x|^\rho}\E\left[\sup_{s\in[0,t]}|X^{n,x}_s|^\rho\right]\le Ce^{Ct}.\label{estimom}
\end{align}
For $0\le r\le u\le s\le t$, writing $X^{n,x}_u-X^{n,x}_r=\int_r^u
a^{1/2}(X^{n,x}_{\lfloor n v/t\rfloor t/n})dW_v$ and using the Burkholder-Davis-Gundy inequality and the growth assumption on $a$, we obtain that
\begin{equation}
   \forall \rho\in[2,+\infty),\;\exists C<\infty,\;\forall n\in\N^*,\;\forall 0\le r\le s\le t,\;\E\left[\sup_{u\in[r,s]}|X^{n,x}_u-X^{n,x}_r|^\rho\right]\le C(s-r)^{\rho/2}.\label{majaccroiss}
\end{equation}
With Kolmogorov's tightness criterion (see for instance \cite[Exercise 4.11 Chapter 2]{KaShre}), we deduce that the sequence $({\cal L}((X^{n,x}_s)_{s\in[0,t]}))_{n\in\N^*}$ of laws of $(X^{n,x}_s)_{s\in[0,t]}$ is tight. With the uniform integrability consequence of the estimate $\sup_{n\in\N^*}\E\left[\sup_{s\in[0,t]}|X^{n,x}_s|^{\rho+1}\right]<\infty$ and \eqref{cvwr}, we deduce that this sequence is relatively compact in ${\cal P}_\rho(C([0,t],\R^d))$ for any $\rho\in [2,+\infty)$ and even for any $\rho\in [1,+\infty)$ since ${\cal W}_\rho$ is non-decreasing in $\rho$.
 Since for $0\le k\le n-1$,
 \begin{align*}
   \sup_{s\in[k t/n,(k +1)t/n]}\Big|i_n\left((X^{n,x}_{\ell t/n})_{\ell =0:n}\right)(s)-X^{n,x}_s\Big|\le 2\sup_{s\in[kt/n,(k+1)t/n]}|X^{n,x}_s-X^{n,x}_{kt/n}|,
 \end{align*}
 we obtain using \eqref{majaccroiss} that for $\rho\in[2,+\infty)$,
 \begin{align}
   \E\left[\sup_{s\in[0,t]}\Big|i_n\left((X^{n,x}_{k t/n})_{k =0:n}\right)(s)-X^{n,x}_s\Big|^\rho\right]\le 2^\rho\E\left[\sum_{k=0}^{n-1}\sup_{s\in[kt/n,(k+1)t/n]}|X^{n,x}_s-X^{n,x}_{kt/n}|^\rho\right]&\le 
 Cn^{1-\rho/2}  .\label{majosupac}
 \end{align}
 With the monotonicity of ${\cal W}_\rho$ in $\rho$, we deduce that $$\forall\rho\in [1,+\infty),\;\lim_{n\to\infty}{\cal W}_\rho({\cal L}((X^{n,x}_s)_{s\in[0,t]}),P^{n,x}_{\le t})=0.$$ Since the sequence $({\cal L}((X^{n,x}_s)_{s\in[0,t]}))_{n\in\N^*}$ is relatively compact in ${\cal P}_\rho(C([0,t],\R^d))$, this implies that so is $(P^{n,x}_{\le t})_{n\in\N^*}$. Moreover, to prove that the limit of any converging subsequence of $(P^{n,x}_{\le t})_{n\in\N^*}$ solves the martingale problem $(MP^x_{\le t})$, it is enough to check the same for any converging subsequence of $({\cal L}((X^{n,x}_s)_{s\in[0,t]}))_{n\in\N^*}$, which we now do.
 For $\varphi:\R^d\to\R$, ${\cal C}^\infty$ with compact support, $p\in\N^*$, $g:(\R^d)^p\to\R$ continuous and bounded, $0\le r_1\le r_2\le \cdots\le r_p\le r\le s\le t$, the function
 $$F(\xi)=\left(\varphi(\xi_s)-\varphi(\xi_r)-\frac 12\int_r^s{\rm tr}\left(a(\xi_u)\nabla^2\varphi(\xi_u)\right)du\right)g(\xi_{r_1},\cdots,\xi_{r_p})$$
 is continuous and bounded on $C([0,T],\R^d)$. Therefore it is enough to check that $\lim_{n\to\infty}\E[F(X^{n,x})]=0$. For $n\in\N^*$, defining $F_n$ like $F$ with the factor $a(\xi_u)$ replaced by $a(\xi_{\lfloor nu/t\rfloor t/n})$ in the trace, we have $\E[F_n(X^{n,x})]=0$.
 Since
 \begin{align*}
   |F_n(X^{n,x})-F(X^{n,x})|\le C\sup_{u\in[0,t]}\Big(&|a\nabla^2\varphi(X^{n,x}_u)-a\nabla^2\varphi(X^{n,x}_{\lfloor nu/t\rfloor t/n})|\\&+|a(X^{n,x}_{\lfloor nu/t\rfloor t/n})||\nabla^2\varphi(X^{n,x}_u)-\nabla^2\varphi(X^{n,x}_{\lfloor nu/t\rfloor t/n})|\Big),
 \end{align*}
 by the Cauchy--Schwarz inequality, the growth assumption on $a$ and \eqref{estimom},
 \begin{align*}
   &|\E[F(X^{n,x})]|=|\E[F(X^{n,x})-F_n(X^{n,x})]|\le \E[|F_n(X^{n,x})-F(X^{n,x})|]\\&\le C\E\left[\sup_{u\in[0,t]}|a\nabla^2\varphi(X^{n,x}_u)-a\nabla^2\varphi(X^{n,x}_{\lfloor nu/t\rfloor t/n})|\right]+C\E^{1/2}\left[\sup_{u\in[0,t]}|\nabla^2\varphi(X^{n,x}_u)-\nabla^2\varphi(X^{n,x}_{\lfloor nu/t\rfloor t/n})|^2\right].
 \end{align*}
 The right-hand side converges to $0$ as $n\to\infty$ since the compactly supported continuous functions $\nabla^2\varphi$ and  $a\nabla^2\varphi$ are uniformly continuous and bounded and, by \eqref{majosupac}, $\sup_{u\in[0,t]}|X^{n,x}_u-X^{n,x}_{\lfloor nu/t\rfloor t/n}|$ converges in probability to $0$. 
\hfill$\Box$
\bibliographystyle{plain}\bibliography{propconvex}\end{document}